\documentclass[11pt,letterpaper]{article}
\usepackage[T1]{fontenc}
\usepackage{amsmath,amsthm,amssymb}
\usepackage{newtxtext,newtxmath}
\usepackage[margin=1in]{geometry}
\usepackage{array,booktabs,enumitem,float,url}
\usepackage{microtype,titlesec,caption,xcolor}
\usepackage[authoryear,round]{natbib}
\definecolor{linkblue}{RGB}{38,67,95}
\usepackage[colorlinks,citecolor=linkblue,linkcolor=linkblue,urlcolor=linkblue]{hyperref}
\hypersetup{pdftitle={Bernstein--von Mises theorem for sparse generalized linear models},
  pdfauthor={Hanqing Li and Xuewen Lu},bookmarksnumbered=true}

\titleformat{\section}{\large\bfseries}{\thesection}{.65em}{}
\titleformat{\subsection}{\normalsize\bfseries}{\thesubsection}{.65em}{}
\titleformat{\subsubsection}{\normalsize\itshape}{\thesubsubsection}{.65em}{}
\titlespacing*{\section}{0pt}{14pt plus 3pt minus 2pt}{6pt plus 1pt}
\titlespacing*{\subsection}{0pt}{10pt plus 2pt minus 2pt}{4pt plus 1pt}
\titlespacing*{\subsubsection}{0pt}{8pt plus 2pt minus 1pt}{3pt plus 1pt}

\makeatletter
\renewcommand{\maketitle}{%
  \begin{center}
  {\LARGE\bfseries\@title\par}\vspace{10pt}
  {\large\@author\par}
  \end{center}\vspace{4pt}}
\makeatother
\renewenvironment{abstract}
  {\begin{list}{}{\setlength{\leftmargin}{.18in}\setlength{\rightmargin}{.18in}%
    \setlength{\topsep}{4pt}\setlength{\parsep}{0pt}}\item[]\small
    \noindent\textbf{Abstract.}\enspace\ignorespaces}
  {\end{list}}

\setlist{topsep=3pt plus 1pt minus 1pt,itemsep=1pt plus .5pt,parsep=0pt,partopsep=0pt}

\AtBeginDocument{%
  \setlength{\abovedisplayskip}{5pt plus 2pt minus 2pt}%
  \setlength{\belowdisplayskip}{5pt plus 2pt minus 2pt}%
  \setlength{\abovedisplayshortskip}{3pt plus 2pt minus 2pt}%
  \setlength{\belowdisplayshortskip}{3pt plus 2pt minus 2pt}%
  \setlength{\textfloatsep}{10pt plus 2pt minus 2pt}%
  \setlength{\floatsep}{8pt plus 2pt minus 2pt}%
  \setlength{\intextsep}{8pt plus 2pt minus 2pt}%
}

\newsavebox{\arxivtablebox}
\newlength{\tabledim}
\newcommand{\tbl}[2]{%
  \sbox{\arxivtablebox}{#2}%
  \setlength{\tabledim}{\dimexpr\textwidth/2-\wd\arxivtablebox/2\relax}%
  \captionsetup{width=\wd\arxivtablebox}%
  \caption{#1}%
  \noindent\makebox[\textwidth][c]{\usebox{\arxivtablebox}}\par
}

\newtheoremstyle{preprintplain}{6pt plus 1pt minus 1pt}{6pt plus 1pt minus 1pt}
  {\itshape}{}{\bfseries}{.}{.5em}{}
\newtheoremstyle{preprintdefinition}{6pt plus 1pt minus 1pt}{6pt plus 1pt minus 1pt}
  {\normalfont}{}{\bfseries}{.}{.5em}{}
\theoremstyle{preprintplain}
\newtheorem{lemma}{Lemma}
\newtheorem{corollary}{Corollary}
\newtheorem{theorem}{Theorem}
\newtheorem{proposition}{Proposition}
\theoremstyle{preprintdefinition}
\newtheorem{assumption}{Assumption}
\newtheorem{remark}{Remark}
\newtheorem{example}{Example}
\DeclareMathOperator*{\argmax}{arg\,max}
\newcommand{\dif}{\,\mathrm{d}}
\newcommand{\mX}{\mathrm{X}}
\newcommand{\mF}{\mathrm{F}}
\newcommand{\mW}{\mathrm{W}}
\newcommand{\mK}{\mathrm{K}}

\newcommand{\mI}{\mathrm{I}}
\newcommand{\mP}{\mathbb{P}}
\newcommand{\mE}{\mathbb{E}}
\newcommand{\TV}{\mathrm{TV}}

\newcommand{\e}{\mathrm e}
\newcommand{\psiEff}{\psi_{n,\mathrm{ef}}}
\newcommand{\cS}{\mathcal S}
\newcommand{\cB}{\mathcal B}

\title{Bernstein--von Mises theorem for sparse\\generalized linear models}
\author{\href{https://orcid.org/0009-0001-0493-237X}{Hanqing Li}\qquad
  \href{https://orcid.org/0000-0002-0905-2697}{Xuewen Lu}\\[4pt]
  {\small Department of Mathematics and Statistics, University of Calgary}\\[2pt]
  {\small\href{mailto:hanqing.li@ucalgary.ca}{hanqing.li@ucalgary.ca}\qquad
  \href{mailto:xlu@ucalgary.ca}{xlu@ucalgary.ca}}}
\date{}

\begin{document}
\maketitle

\begin{abstract}
We establish an oracle Bernstein--von Mises theorem for high-dimensional sparse generalized linear models under an outcome-independent spike-and-slab prior. The posterior consistently recovers the true active set and converges in total variation to the Gaussian law that would arise if this set were known in advance, even when the active dimension grows. The result holds for both ordinary and fractional posteriors, the latter with the expected variance inflation under tempering. In fixed-design logistic regression, a sufficient lower bound on the squared minimum signal is proportional to the logarithm of the ambient dimension divided by sample size. Separate likelihood controls for model selection and oracle approximation allow the Gaussian approximation to use a weaker curvature condition. We verify the conditions for six canonical and noncanonical models under fixed and random designs, including Poisson regression with potentially unbounded Fisher weights. The results also yield oracle-rate contraction, credible regions on the recovered support and pointwise frequentist coverage.
\end{abstract}

\begingroup\small
\par
\endgroup

\section{Introduction}

The Bernstein--von Mises (BvM) theorem connects Bayesian uncertainty quantification with frequentist inference by approximating the posterior with a Gaussian distribution centred at an efficient estimator \citep{le1986asymptotic,van1998asymptotic}. Classical theory is formulated for regular models with fixed dimension. When the dimension grows, establishing the Gaussian approximation becomes more difficult. Existing extensions include increasing-dimensional Gaussian regression and finite-sample semiparametric bounds \citep{bontemps2011bernstein,panov2015finite}. Sparse high-dimensional models introduce a further complication because the active coordinates are unknown and must be selected. The inferential problem therefore has two parts. The posterior must identify the active set and must quantify uncertainty on that set as accurately as if it had been known in advance.

The requirement that a procedure performs as well as one that knows the active set is expressed by the classical oracle property \citep{fan2001variable}. In the Bayesian setting, \citet{jiang2019oracle} developed a corresponding posterior oracle property and showed that model-selection consistency implies this property. An oracle BvM theorem additionally requires a Gaussian approximation on the true model, possibly with growing active dimension. Spike-and-slab priors make the support an explicit posterior variable, with established contraction and selection theory in Gaussian regression \citep{narisetty2014bayesian,castillo2015bayesian,martin2017empirical,belitser2020empirical}. For generalized linear models (GLMs), earlier work addresses many-covariate normal approximations \citep{ghosal1997normal}, sparse estimation and model selection \citep{jiang2007bayesia,liang2013bayesian,abramovich2016model}, and contraction for canonical and noncanonical links \citep{jeong2021posterior}. A recent review is given by \citet{banerjee2026bayesian}.

At the theorem level, \citet{tang2024empirical} is the closest comparator. Their result combines exact selection with a growing-dimensional oracle Gaussian approximation for canonical and noncanonical links. They use an empirical prior centred at support-specific maximum-likelihood estimators and scaled by local information. Their analysis assumes a fractional power $\alpha<1$. The related empirical-prior analysis of \citet{lee2025advances} also assumes $\alpha<1$. It is restricted to canonical links and does not establish a growing-dimensional Gaussian approximation. The outcome-independent conjugate construction of \citet{filippozzi2026bayesian} covers selection for growing-dimensional canonical GLMs, but does not establish a growing-active-dimension Gaussian approximation or cover fractional posteriors and noncanonical links. Known-support growing-dimensional GLM BvM theory is available \citep{katsevich2025improved}, but it does not address model uncertainty or exact support recovery. These analyses do not provide an oracle BvM theorem for ordinary, outcome-independent spike-and-slab posteriors across canonical and noncanonical GLMs.

We establish an oracle Bernstein--von Mises theorem for sparse GLMs under a spike-and-slab prior chosen independently of the observed outcomes. The posterior recovers the true active set and has the same asymptotic Gaussian law as if this set were known in advance, even when its size grows. The result covers canonical and noncanonical links and applies to both ordinary and fractional posteriors. For fixed-design logistic regression, our sufficient squared minimum-signal bound is of order $\log p/n$. In the comparison in Table~\ref{tab:closest-work}, this improves on \citet{tang2024empirical} and matches the order in \citet{lee2025advances}. For this logistic benchmark, the sufficient sparsity condition $s_0^2\log p=o(n)$ is weaker than the cubic condition $s_0^3\log p=o(n)$ in both studies. We verify the conditions for six models under fixed and random designs, including Poisson regression with potentially unbounded Fisher weights. The results also yield oracle-rate contraction and credible regions, with pointwise frequentist coverage under the stated sampling-calibration conditions.

\begin{table}[!t]
\centering
\begingroup
\setlength{\tabcolsep}{1.5pt}
\renewcommand{\arraystretch}{1.0}
\tbl{Signal and sparsity conditions for fixed-design logistic regression.}{%
\footnotesize
\begin{tabular}{@{}
>{\raggedright\arraybackslash}p{0.14\textwidth}
>{\raggedright\arraybackslash}p{0.24\textwidth}
>{\raggedright\arraybackslash}p{0.18\textwidth}
>{\raggedright\arraybackslash}p{0.21\textwidth}
>{\raggedright\arraybackslash}p{0.19\textwidth}@{}}
\toprule
Work & prior and power & signal threshold & sparsity condition & Gaussian limit\\
\midrule
\citet{tang2024empirical}
& empirical Gaussian; $\alpha\in(0,1)$
& $b_0^2\gtrsim s_0\log p/n$
& $s_0^3\log p=o(n)$
& $\mathcal N(\widehat\beta,\varrho_{\mathrm{TM}}\widehat J^{-1})$\\
\addlinespace[2pt]
\citet{lee2025advances}
& empirical Gaussian; $\alpha\in(0,1)$
& $b_0^2\gtrsim\log p/n$
& $s_0^3\log p=o(n)$
& not established\\
\addlinespace[2pt]
This paper
& outcome-independent slab; $\alpha\in(0,1]$
& $b_0^2\gtrsim\log p/n$
& $s_0^2\log p=o(n)$
& $\mathcal N(\widehat\beta,\alpha^{-1}\widehat J^{-1})$\\
\bottomrule
\end{tabular}}
\label{tab:closest-work}

\vspace{1pt}
\parbox{\dimexpr\textwidth-2\tabledim\relax}{\scriptsize $n,p,s_0,b_0$: sample size, dimension, sparsity and minimum signal. $\widehat\beta,\widehat J$: oracle MLE and observed information; $\varrho_{\mathrm{TM}}$: Tang--Martin covariance multiplier. See Supplement~\ref{supp-sec:logistic-benchmark-comparison}.}
\endgroup
\end{table}

\section{Model and prior}
\label{sec:model-prior}

Let $Y=(Y_1,\ldots,Y_n)^\top$ be independent observations with density
\begin{equation}\label{eq:exponential-family}
f_{i,\beta}(y_i)=\exp\left[\frac{y_i\theta_i-b(\theta_i)}{\tau_i}+k(y_i,\tau_i)\right],
\end{equation}
where $\tau_i>0$ is known and $b$ is strictly convex on the interior $\Theta$ of its natural-parameter domain. Let $\mathcal M=b'(\Theta)$ be the mean domain. We assume that the response-scale link $h:\mathcal M\to\mathbb R$ is a three times continuously differentiable bijection with $h'>0$ on $\mathcal M$, and put $\xi=(b')^{-1}\circ h^{-1}:\mathbb R\to\Theta$. Thus every linear predictor considered below determines a valid natural parameter. Let $X_i\in\mathbb R^p$ be fixed predictors, let $\mX=(X_1,\ldots,X_n)^\top$, $\eta_i=X_i^\top\beta$, and $\theta_i=\xi(\eta_i)$. The canonical case is $\xi(\eta)=\eta$. Put $K(Y)=\sum_{i=1}^n k(Y_i,\tau_i)$. The log likelihood is
\[
\ell_n(\beta)=\sum_{i=1}^n\frac{Y_i\xi(X_i^\top\beta)-b\{\xi(X_i^\top\beta)\}}{\tau_i}+K(Y).
\]

The data are generated from \eqref{eq:exponential-family} at a correctly specified parameter $\beta^0\in\mathbb R^p$. Throughout the fixed-design theory, $\mP_{\beta^0}$ and $\mE_{\beta^0}$ denote probability and expectation under the joint sampling law of $Y$ in \eqref{eq:exponential-family} at $\beta^0$, conditional on $\mX$. For any $\beta$, write $S_\beta=\{j:\beta_j\ne0\}$. The true support is $S_0=\{j:\beta_j^0\ne0\}$, its sparsity is $s_0=|S_0|$, and its minimum signal is $b_0=\min_{j\in S_0}|\beta_j^0|$. For $S\subseteq\{1,\ldots,p\}$, let $\beta_S$ denote the corresponding coordinate vector and let $\widetilde\beta_S\in\mathbb R^p$ be its zero-padded extension, defined by $(\widetilde\beta_S)_S=\beta_S$ and $(\widetilde\beta_S)_{S^c}=0$. Write $\ell_{n,S}(\beta_S)=\ell_n(\widetilde\beta_S)$ and, whenever the maximizer is unique, define $\widehat\beta_S=\argmax_{\beta_S}\ell_{n,S}(\beta_S)$ and $\widehat J_S=-\nabla_S^2\ell_{n,S}(\widehat\beta_S)$.

The score at the truth is $\Delta=\nabla\ell_n(\beta^0)$, and the corresponding Fisher information is $\mF^0=-\mE_{\beta^0}\nabla^2\ell_n(\beta^0)=\mX^\top\mW_0\mX$, where
\[
\mW_0
=\operatorname{diag}\left[
\frac{b''\{\xi(\eta_i^0)\}\{\xi'(\eta_i^0)\}^2}{\tau_i}
\right]_{i=1}^n,
\qquad \eta_i^0=X_i^\top\beta^0.
\]
For index sets $A,B\subseteq\{1,\ldots,p\}$, let $\Delta_A$ be the corresponding subvector of $\Delta$, put $\mF_A^0=\mF^0|_{A\times A}$, and write $\mF_{AB}^0=\mF^0|_{A\times B}$ for a cross-block.

For any index set $A$, let $\delta_A$ denote the Dirac probability measure at $0_A\in\mathbb R^{|A|}$. We use a spike-and-slab prior \citep{mitchell1988bayesian,george1993variable,johnson2012bayesian}. First draw a support size $s$ from a mass function $\pi_p$ and then draw $S$ uniformly among the $\binom ps$ supports of that size. Finally, draw $\beta_S$ from a density $g_S$ and set $\beta_{S^c}=0$. Thus
\[
\Pi_n(S,\mathrm d\beta)
=\frac{\pi_p(|S|)}{\binom p{|S|}}
g_S(\beta_S)\dif\beta_S\,\delta_{S^c}(\mathrm d\beta_{S^c}).
\]

For a fixed $\alpha\in(0,1]$, the powered posterior is
\[
\Pi_{n,\alpha}(B\mid Y)
=\frac{\int_B\exp\{\alpha\ell_n(\beta)\}\Pi_n(\mathrm d\beta)}
{\int\exp\{\alpha\ell_n(\beta)\}\Pi_n(\mathrm d\beta)}.
\]
The ordinary posterior is the endpoint $\alpha=1$. For $\alpha<1$, tempering reduces the influence of the likelihood relative to the prior \citep{bhattacharya2019bayesian,alquier2020concentration}.
For a support $S$, define the powered marginal likelihood
\[
\mathcal Z_{n,\alpha}(S)
=\int_{\mathbb R^{|S|}}\e^{\alpha\ell_{n,S}(\beta_S)}g_S(\beta_S)\dif\beta_S.
\]
Then $\Pi_{n,\alpha}(S_\beta=S\mid Y)$ is proportional to
$\pi_p(|S|)\mathcal Z_{n,\alpha}(S)/\binom p{|S|}$.

\section{Assumptions}
\label{sec:assumptions}

Fix a sufficiently large integer $K_0>1$, put $s_n^\dagger=K_0s_0$ and $s_n^\star=s_n^\dagger+s_0$, and let $\cS_n^+=\{S:S\supseteq S_0,\ |S|\le s_n^\star\}$. For $r>0$, write $\cB_S(r)=\{h\in\mathbb R^{|S|}:\|(\mF_S^0)^{1/2}h\|_2\le r\}$. For $S\in\cS_n^+$, Assumption~\ref{ass:sparse-fisher} ensures that this is the radius-$r$ ball centred at zero in the Fisher norm $h\mapsto\|(\mF_S^0)^{1/2}h\|_2$. Throughout the remainder of the paper, $s_0\ge1$, $\log p\to\infty$, and $\log n=O(\log p)$.

Assumptions~\ref{ass:sparse-fisher}--\ref{ass:truth-lan} control the likelihood over sparse supports, and Assumptions~\ref{ass:prior}--\ref{ass:signal} give prior and signal conditions for exact selection. Provided $\pi_p(s_0)>0$, Proposition~\ref{prop:conditional-bvm} obtains the conditional Gaussian approximation from Assumptions~\ref{ass:sparse-fisher}--\ref{ass:truth-lan}, \ref{ass:slab} and \ref{ass:oracle-laplace}. Theorem~\ref{thm:oracle-bvm} combines the two conclusions.

\subsection{Selection geometry}

\begin{assumption}[Sparse Fisher geometry]
\label{ass:sparse-fisher}
The support budget satisfies $s_n^\star=o(p)$, and for some fixed $0<c_F<C_F<\infty$,
\[
c_F\le\rho_{\min}(n^{-1}\mF_S^0)
\le\rho_{\max}(n^{-1}\mF_S^0)\le C_F,
\qquad S\in\cS_n^+.
\]
\end{assumption}

\begin{assumption}[Bernstein score tails]
\label{ass:score}
For some sequence $d_n\ge0$, $d_n\sqrt{s_n^\star\log p}=o(1)$ and, for every $S\in\cS_n^+$, $u\in\mathbb S^{|S|-1}$, and $\lambda$ satisfying $|\lambda|d_n<1$,
\[
\log\mE_{\beta^0}
\exp\{\lambda u^\top(\mF_S^0)^{-1/2}\Delta_S\}
\le \frac{\lambda^2}{2(1-|\lambda|d_n)}.
\]
\end{assumption}

Assumption~\ref{ass:score} is a uniform Bernstein condition in Fisher coordinates. The choice $d_n=0$ gives the sub-Gaussian case. Section~\ref{sec:glm-examples} verifies the condition under explicit design assumptions, and the supplement gives the Fisher-standardized score bounds.

\begin{assumption}[Selection-scale curvature]
\label{ass:truth-lan}
For every $S\in\cS_n^+$, $\ell_{n,S}$ is samplewise concave on $\mathbb R^{|S|}$.
There exist $r_n^0>0$ and $\varepsilon_n\downarrow0$ such that $r_n^0/\sqrt{s_n^\star\log p}\to\infty$ and $\varepsilon_n(s_0+\log p)=o(\log p)$, and, with $\mP_{\beta^0}$-probability tending to one,
\[
\sup_{S\in\cS_n^+}\sup_{h\in\cB_S(r_n^0)}
\left\|(\mF_S^0)^{-1/2}
\{-\nabla_S^2\ell_{n,S}(\beta_S^0+h)-\mF_S^0\}
(\mF_S^0)^{-1/2}\right\|_{\rm op}
\le\varepsilon_n.
\]
\end{assumption}

Assumption~\ref{ass:truth-lan} requires only a relative-information approximation on the larger selection neighbourhood. Together with Assumptions~\ref{ass:sparse-fisher} and \ref{ass:score}, it also gives uniform score localization and radial likelihood decay. The condition $\varepsilon_n(s_0+\log p)=o(\log p)$ makes the accumulated information-sandwich error negligible relative to the coordinatewise support cost.

\subsection{Prior and identifiability}

\begin{assumption}[Model-size prior]
\label{ass:prior}
For some fixed $c_-,c_+>0$ and $a_-\ge a_+>1$, with $K_0$ sufficiently large,
\[
c_-p^{-a_-}\pi_p(s-1)
\le \pi_p(s)
\le c_+p^{-a_+}\pi_p(s-1),
\qquad 1\le s\le p.
\]
\end{assumption}

The recursion is standard for sparse priors \citep{castillo2015bayesian} and includes $\pi_p(s)\propto p^{-as}$ as the special case $a_-=a_+=a$. In the supplementary proof, slab normalization and fractional moments of posterior-odds bounds control the sum over added coordinates for $a_+>1$. This moment argument applies to both ordinary and fractional posteriors. After fixing $a_+$, we choose $K_0$ sufficiently large to control the posterior probability of supports larger than $K_0s_0$.

\begin{assumption}[Slab regularity]
\label{ass:slab}
The slab densities have the product form $g_S(\beta_S)=\prod_{j\in S}g_j(\beta_j)$, where the $g_j$ are positive probability densities on $\mathbb R$. For some fixed $C_g>0$,
\[
\min_{1\le j\le p}g_j(\beta_j^0)\ge p^{-C_g},
\qquad
\sup_{S\in\cS_n^+}\sup_{\|u\|_2\le r_n^0/\sqrt{nc_F}}
\left|\log\frac{g_S(\beta_S^0+u)}{g_S(\beta_S^0)}\right|=O(1).
\]
\end{assumption}

Assumption~\ref{ass:slab} requires sufficient slab density at the true coefficients and bounded density variation on the selection neighbourhoods. Slab normalization controls the density at zero together with the determinant in the marginal likelihood, so no separate bound at zero is imposed. Gaussian, Laplace and Student-$t$ examples are verified in Section~\ref{sec:slab-examples}.

\begin{assumption}[Beta-min condition]
\label{ass:signal}
For the fixed $\alpha\in(0,1]$ under consideration, $nb_0^2\ge C_\alpha\log p$ for a sufficiently large constant $C_\alpha>0$.
\end{assumption}

The constant $C_\alpha$ may depend on the fixed constants in Assumptions~\ref{ass:sparse-fisher}--\ref{ass:slab}, in particular on $a_-$ and hence on $a$ when $\pi_p(s)\propto p^{-as}$. Assumption~\ref{ass:signal} is needed only to exclude underfitted supports. Our signal condition does not impose an additional uniform bound on the absolute row sums of $(n^{-1}\mF_S^0)^{-1}$. In \citet{lee2025advances}, attaining the same beta-min order through the coordinatewise MLE bound requires bounded inverse-information and score factors; see their equations (5.17) and (5.20). Our comparison of omitted-signal losses avoids this coordinatewise estimation step.

\subsection{Oracle approximation}

Assumptions~\ref{ass:sparse-fisher}--\ref{ass:truth-lan} imply that the oracle maximum-likelihood estimator exists uniquely and has positive observed information with probability tending to one. Its truth-Fisher distance is $O_{\mP_{\beta^0}}(\sqrt{s_0})$; see Lemma~\ref{supp-le:oracle-mle}. On the exceptional event, set $\widehat\beta_{S_0}=0$ and $\widehat J_{S_0}=\mI_{s_0}$.

\begin{samepage}
\begin{assumption}[Oracle approximation]
\label{ass:oracle-laplace}
There exist $r_n>0$ and a deterministic sequence $\bar\kappa_n^0\ge0$ such that $r_n/\sqrt{s_0}\to\infty$ and $s_0\bar\kappa_n^0\to0$, and, with $\mP_{\beta^0}$-probability tending to one,
\[
\kappa_n^0=
\sup_{0<\|\widehat J_{S_0}^{1/2}u\|_2\le r_n}
\frac{\left\|\widehat J_{S_0}^{-1/2}
\{-\nabla_{S_0}^2\ell_{n,S_0}(\widehat\beta_{S_0}+u)
-\widehat J_{S_0}\}
\widehat J_{S_0}^{-1/2}\right\|_{\rm op}}
{\|\widehat J_{S_0}^{1/2}u\|_2}
\le\bar\kappa_n^0.
\]
In addition, for every fixed $M>0$,
\[
\sup_{\|u\|_2\le M\sqrt{s_0/n}}
\left|\log\frac{g_{S_0}(\beta_{S_0}^0+u)}
{g_{S_0}(\beta_{S_0}^0)}\right|\longrightarrow0.
\]
\end{assumption}
\end{samepage}

Assumption~\ref{ass:oracle-laplace} controls likelihood curvature and prior flatness only on the true support. The integrated Laplace bound of \citet{katsevich2025improved} gives a leading total-variation error of order $s_0\kappa_n^0$. Its radius requirement is met by reducing the radius within the proof. The slab need be asymptotically flat only on the smaller oracle neighbourhoods, while bounded variation in Assumption~\ref{ass:slab} suffices for support comparison. Fisher standardization is also used in related sparse-GLM analyses \citep{tang2024empirical,lee2025advances}.

\section{Selection consistency}
\label{sec:selection}

For an underfitted support $S\nsupseteq S_0$ with $|S|\le s_n^\dagger$, let $M=S_0\setminus S$ and $T=S\cup S_0$. Within model $T$, model $S$ fixes the omitted coordinates $M$ at zero. The information for detecting these coefficients after accounting for $S$ is $\mK_{M\mid S}=\mF_M^0-\mF_{MS}^0(\mF_S^0)^{-1}\mF_{SM}^0$.
Assumption~\ref{ass:sparse-fisher} and the definition of $b_0$ give
\[
(\beta_M^0)^\top\mK_{M\mid S}\beta_M^0
\ge c_Fn\|\beta_M^0\|_2^2
\ge c_Fn|M|b_0^2.
\]
Assumption~\ref{ass:score} controls the projected score fluctuations that could offset this loss, while the model-size prior and the same bound control additional noise coordinates. Together with posterior-dimension and radial-tail bounds, these give selection consistency.

\begin{theorem}[Selection consistency]\label{thm:selection}
Fix $\alpha\in(0,1]$. Under Assumptions~\ref{ass:sparse-fisher}--\ref{ass:signal},
\[
\Pi_{n,\alpha}(S_\beta=S_0\mid Y)
\xrightarrow{\mP_{\beta^0}}1.
\]
\end{theorem}

A sufficiently large $C_\alpha$ in $nb_0^2\ge C_\alpha\log p$ makes the missing-signal loss dominate the per-coordinate prior, determinant, slab and support-counting costs.

\begin{remark}[Outcome-independent marginal likelihoods]
\label{rem:marginal-likelihoods}
For $T\in\cS_n^+$, put $Z_T=(\mF_T^0)^{-1/2}\Delta_T$ and, for $c>0$, define
\[
\mathcal J_{n,\alpha}(T,c)
=\left(\frac{2\pi}{\alpha c}\right)^{|T|/2}
\exp\left\{\frac{\alpha}{2c}\|Z_T\|_2^2\right\}.
\]
Under Assumptions~\ref{ass:sparse-fisher}--\ref{ass:truth-lan} and \ref{ass:slab}, for fixed $\alpha\in(0,1]$, there is a fixed $C<\infty$ such that, with $\mP_{\beta^0}$-probability tending to one, uniformly over $T\in\cS_n^+$,
\begin{equation}
\begin{aligned}
C^{-1}\mathcal J_{n,\alpha}(T,1+\varepsilon_n)
&\le \frac{\mathcal Z_{n,\alpha}(T)}
{\e^{\alpha\ell_n(\beta^0)}g_T(\beta_T^0)|\mF_T^0|^{-1/2}}\\
&\le C\mathcal J_{n,\alpha}(T,1-\varepsilon_n).
\end{aligned}
\label{eq:marginal-likelihood-bounds}
\end{equation}
These bounds make the slab-density and determinant terms explicit. In \citet{tang2024empirical} and \citet{lee2025advances}, fitted-mode centring and local-information scaling incorporate the corresponding terms into the posterior support weights. Here, when a support adds noise coordinates to $S_0$, the common slab factors cancel, and slab normalization controls the remaining density at zero together with the determinant. For underfitted supports, an affine-slice comparison exposes the loss from omitting true coefficients. Section~\ref{supp-sec:selection-proof} proves \eqref{eq:marginal-likelihood-bounds} and the resulting support comparisons.
\end{remark}

\section{Bernstein--von Mises theorem}
\label{sec:bvm}

We first establish a Gaussian approximation conditionally on the true support, without beta-min separation or selection consistency, and then combine it with support recovery.
Write $\|\cdot\|_{\rm TV}$ for total variation distance.

\begin{proposition}[Conditional Gaussian approximation]\label{prop:conditional-bvm}
Fix $\alpha\in(0,1]$. Suppose $\pi_p(s_0)>0$ and Assumptions~\ref{ass:sparse-fisher}--\ref{ass:truth-lan}, \ref{ass:slab} and \ref{ass:oracle-laplace} hold. Then
\[
\left\|
\Pi_{n,\alpha}(\beta_{S_0}\in\cdot\mid S_\beta=S_0,Y)
-\mathcal N_{s_0}\big(\widehat\beta_{S_0},(\alpha\widehat J_{S_0})^{-1}\big)
\right\|_{\rm TV}
\xrightarrow{\mP_{\beta^0}}0.
\]
\end{proposition}

Here $\Pi_{n,\alpha}(\beta_{S_0}\in\cdot\mid S_\beta=S_0,Y)$ denotes the conditional posterior of the active coordinates on $\mathbb R^{s_0}$. The Gaussian law in the full result below is embedded in $\mathbb R^p$ by placing a point mass at zero on $S_0^c$.

\begin{theorem}[Bernstein--von Mises]\label{thm:oracle-bvm}
Fix $\alpha\in(0,1]$. Under Assumptions~\ref{ass:sparse-fisher}--\ref{ass:oracle-laplace},
\[
\left\|
\Pi_{n,\alpha}(\,\cdot\mid Y)
-
\mathcal N_{s_0}
\big(\widehat\beta_{S_0},
(\alpha\widehat J_{S_0})^{-1}\big)
\otimes\delta_{S_0^c}
\right\|_{\rm TV}
\xrightarrow{\mP_{\beta^0}}0.
\]
\end{theorem}

Theorem~\ref{thm:selection} and Proposition~\ref{prop:conditional-bvm} imply Theorem~\ref{thm:oracle-bvm} because its total-variation error is bounded by
\[
\Pi_{n,\alpha}(S_\beta\ne S_0\mid Y)
+\left\|\Pi_{n,\alpha}(\beta_{S_0}\in\cdot\mid S_\beta=S_0,Y)-\mathcal N_{s_0}\big(\widehat\beta_{S_0},(\alpha\widehat J_{S_0})^{-1}\big)\right\|_{\rm TV}.
\]
Incorrect supports enter only through their aggregate posterior mass, as in \citet{tang2024empirical}, rather than through a Gaussian-mixture approximation over candidate supports \citep{castillo2015bayesian}.
Expanding the likelihood at the oracle maximum-likelihood estimator gives the observed-information covariance. Score localization and the relative-information bound also give
\[
\|(\mF_{S_0}^0)^{1/2}(\widehat\beta_{S_0}-\beta_{S_0}^0)\|_2=O_{\mP_{\beta^0}}(\sqrt{s_0}),
\qquad
\|\widehat\beta_{S_0}-\beta_{S_0}^0\|_2=O_{\mP_{\beta^0}}(\sqrt{s_0/n}).
\]

\begin{remark}[Classical Fisher-information form]
\label{rem:fisher-information-form}
If, in addition, $\sqrt{s_0}\,\varepsilon_n=o(1)$, then
\[
\left\|
\Pi_{n,\alpha}(\,\cdot\mid Y)
-\mathcal N_{s_0}
\big(\widehat\beta_{S_0},(\alpha\mF_{S_0}^0)^{-1}\big)
\otimes\delta_{S_0^c}
\right\|_{\rm TV}
\xrightarrow{\mP_{\beta^0}}0.
\]
\end{remark}

For fixed $s_0$, the additional condition is automatic because $\varepsilon_n\to0$. For growing $s_0$, it is also verified by the exactly quadratic Gaussian model, by logistic, Poisson and bounded-window probit regression under average spread, and by the separate unbounded-weight Poisson result below. It remains an additional condition for Gamma and negative-binomial regression and for the generic sub-Gaussian random-design bridge.

\subsection{Contraction}

The oracle approximation and maximum-likelihood rate give the following contraction result.

\begin{corollary}[Posterior contraction]
\label{cor:posterior-contraction}
Under conditions of Theorem~\ref{thm:oracle-bvm}, for every deterministic sequence $L_n\to\infty$,
\[
\Pi_{n,\alpha}(\|\beta-\beta^0\|_2>L_n\sqrt{s_0/n}\mid Y)\xrightarrow{\mP_{\beta^0}}0.
\]
\end{corollary}

\begin{remark}[Oracle versus global contraction]
Exact support recovery gives the known-support contraction rate $\sqrt{s_0/n}$. The usual global rate $\{s_0\log p/n\}^{1/2}$ also includes the cost of locating the support \citep{castillo2015bayesian,jeong2021posterior,tang2024empirical,lee2025advances}. The corollary therefore relies on signal separation and does not give the oracle rate for arbitrarily weak signals.
\end{remark}

\subsection{Inferential consequences}
\label{sec:inferential-consequences}

Total-variation convergence gives asymptotically correct posterior content, whereas frequentist coverage requires sampling conditions on the oracle maximum-likelihood estimator. Throughout this subsection, dependence on $n$ is suppressed in the credible-region notation and related quantities. Let $q_{d,1-\gamma}$ denote the $(1-\gamma)$ quantile of $\chi_d^2$ and let $z_{1-\gamma/2}$ denote the corresponding standard-normal quantile. For any support $S$ and $j\in S$, let $e_j\in\mathbb R^S$ be the unit vector selecting coordinate $j$. Its ambient support is clear from context. Define
\begin{equation}\label{eq:oracle-credible-ellipsoid}
\mathcal C_{\alpha}(1-\gamma)=\Big\{\beta:S_\beta=S_0,\ \alpha(\beta_{S_0}-\widehat\beta_{S_0})^\top\widehat J_{S_0}(\beta_{S_0}-\widehat\beta_{S_0})\le q_{s_0,1-\gamma}\Big\}.
\end{equation}
Fix an index $j$ that belongs to $S_0$ eventually and is specified before the data are observed. Write $\widehat\beta_j=e_j^\top\widehat\beta_{S_0}$ and define the oracle marginal interval
\begin{equation}\label{eq:oracle-marginal-interval}
\mathcal I_{j,\alpha}(1-\gamma)=\widehat\beta_j\mathbin{\pm}z_{1-\gamma/2}\{\alpha^{-1}e_j^\top\widehat J_{S_0}^{-1}e_j\}^{1/2}.
\end{equation}
For the sampling calibration, write the normalized oracle score contributions as
\[
\zeta_i
=(\mF_{S_0}^0)^{-1/2}\nabla_{S_0}\log f_{i,\beta^0}(Y_i)
=(\mF_{S_0}^0)^{-1/2}X_{i,S_0}
\frac{\xi'(\eta_i^0)}{\tau_i}(Y_i-\mu_i^0),
\qquad \mu_i^0=b'\{\xi(\eta_i^0)\}.
\]
Correct specification gives independent mean-zero vectors satisfying $\sum_i\operatorname{Var}_{\beta^0}(\zeta_i)=\mI_{s_0}$. For such $j$, let $v_j=(\mF_{S_0}^0)^{-1/2}e_j/\{e_j^\top(\mF_{S_0}^0)^{-1}e_j\}^{1/2}$.

\begin{corollary}[Credible sets and coverage]
\label{cor:coverage}
Under conditions of Theorem~\ref{thm:oracle-bvm}, fix $0<\gamma<1$, $\alpha\in(0,1]$ and an active coordinate $j$. Then
\begin{align*}
\Pi_{n,\alpha}\{\mathcal C_{\alpha}(1-\gamma)\mid Y\}
&=1-\gamma+o_{\mP_{\beta^0}}(1),\\
\Pi_{n,\alpha}\{\beta_j\in\mathcal I_{j,\alpha}(1-\gamma)\mid Y\}
&=1-\gamma+o_{\mP_{\beta^0}}(1).
\end{align*}
If $s_0=d\ge1$ is fixed and the Lyapunov condition $\sup_{\|v\|_2=1}\sum_{i=1}^n\mE_{\beta^0}|v^\top\zeta_i|^{2+\delta}=o(1)$ holds for some $\delta>0$, then
\[
\mP_{\beta^0}\{\beta^0\in\mathcal C_{\alpha}(1-\gamma)\}
\longrightarrow
F_{\chi_d^2}\{q_{d,1-\gamma}/\alpha\}.
\]
For every fixed active coordinate $j$, the marginal limit below holds under the same fixed-dimensional condition. It also holds when $s_0$ grows, provided, for some $\delta>0$, $\sum_{i=1}^n\mE_{\beta^0}|v_j^\top\zeta_i|^{2+\delta}=o(1)$ and $\varepsilon_n\sqrt{s_0}=o(1)$. In either case,
\begin{equation}\label{eq:fractional-marginal-coverage}
\mP_{\beta^0}\{\beta_j^0\in
\mathcal I_{j,\alpha}(1-\gamma)\}
\longrightarrow
2\Phi\{z_{1-\gamma/2}/\sqrt\alpha\}-1.
\end{equation}
\end{corollary}

For the Gaussian, logistic, bounded-window probit, Gamma, fixed-size negative-binomial and upper-window Poisson routes in Section~\ref{sec:glm-examples}, the Lyapunov conditions above hold with $\delta=1$ whenever $\max_i\|(\mF_{S_0}^0)^{-1/2}X_{i,S_0}\|_2=o(1)$. The moment bounds are verified in the supplement. Under the fixed-design criterion below, this maximum is $O\{(s_0/n)^{1/2}\}$, so these conditions are automatic for fixed $s_0$. For growing $s_0$, the marginal limit additionally requires $\varepsilon_n\sqrt{s_0}=o(1)$. The separate unbounded-weight Poisson route in Section~\ref{sec:random-design-examples} still requires an explicit Lyapunov condition.

Define $\widehat S_\alpha=\argmax_S\Pi_{n,\alpha}(S_\beta=S\mid Y)$ as the posterior modal support, with deterministic tie-breaking.
The corresponding selected-support ellipsoid and reported post-selection interval are
\[
\begin{aligned}
\widehat{\mathcal C}_{\alpha}(1-\gamma)
&=\left\{\beta:S_\beta=\widehat S_\alpha,\ 
\alpha(\beta_{\widehat S_\alpha}-\widehat\beta_{\widehat S_\alpha})^\top
\widehat J_{\widehat S_\alpha}
(\beta_{\widehat S_\alpha}-\widehat\beta_{\widehat S_\alpha})
\le q_{|\widehat S_\alpha|,1-\gamma}\right\},\\
\widehat{\mathcal I}_{j,\alpha}(1-\gamma)
&=\begin{cases}
e_j^\top\widehat\beta_{\widehat S_\alpha}\mathbin{\pm}z_{1-\gamma/2}
\{\alpha^{-1}e_j^\top\widehat J_{\widehat S_\alpha}^{-1}e_j\}^{1/2},
&j\in\widehat S_\alpha,\\
\varnothing,&j\notin\widehat S_\alpha.
\end{cases}
\end{aligned}
\]
Here $\varnothing$ means that no interval is reported when $j\notin\widehat S_\alpha$. The set $\widehat{\mathcal I}_{j,\alpha}$ is a post-selection report rather than a full marginal credible set for every coordinate. Take the ellipsoid to be empty and report no interval if the required maximum-likelihood estimator or inverse information is not defined.

\begin{corollary}[Selected-support inference]
\label{cor:selected-coverage}
Under conditions of Theorem~\ref{thm:oracle-bvm},
\[
\mP_{\beta^0}(\widehat S_\alpha=S_0)\longrightarrow1.
\]
The posterior-credibility conclusions and every applicable frequentist-coverage limit in Corollary~\ref{cor:coverage} remain valid after replacing $\mathcal C_{\alpha}$ and $\mathcal I_{j,\alpha}$ by $\widehat{\mathcal C}_{\alpha}$ and $\widehat{\mathcal I}_{j,\alpha}$. In particular, under the corresponding marginal-coverage conditions, the following limit holds both unconditionally and under $\mP_{\beta^0}(\,\cdot\mid j\in\widehat S_\alpha)$, since $\mP_{\beta^0}(j\in\widehat S_\alpha)\to1$,
\[
\mP_{\beta^0}\{j\in\widehat S_\alpha,\,
\beta_j^0\in\widehat{\mathcal I}_{j,\alpha}(1-\gamma)\}
\longrightarrow
2\Phi\{z_{1-\gamma/2}/\sqrt\alpha\}-1.
\]
\end{corollary}

\begin{remark}[Calibration under fractional posteriors]
Under the respective coverage conditions of Corollary~\ref{cor:coverage}, the unadjusted oracle ellipsoid (for fixed $s_0$) and marginal interval have nominal frequentist coverage $1-\gamma$ at $\alpha=1$, but are conservative at every fixed $\alpha<1$. Multiplying the threshold in \eqref{eq:oracle-credible-ellipsoid} by $\alpha$ or removing $\alpha^{-1/2}$ from the half-width in \eqref{eq:oracle-marginal-interval} yields the corresponding nominal Wald region, but changes its fractional-posterior credibility. See \citet{syring2019calibrating} for a broader perspective on posterior calibration.
\end{remark}

\section{Examples}
\label{sec:examples}

We verify the slab and likelihood conditions below. Corollary~\ref{cor:primitive-six-models} combines these with the separately imposed Assumptions~\ref{ass:prior} and \ref{ass:signal} on the model-size prior and signal separation.

\subsection{Slab priors}
\label{sec:slab-examples}

We verify Assumption~\ref{ass:slab} and the oracle slab condition in Assumption~\ref{ass:oracle-laplace} for Gaussian, Laplace and Student-$t$ product slabs under the truth-size and local-radius conditions below. For fixed $\sigma,\lambda,\nu>0$, define
\[
\begin{gathered}
g_{\mathrm G,\sigma}(x)=\frac{1}{\sqrt{2\pi\sigma^2}}
      \exp\left(-\frac{x^2}{2\sigma^2}\right),
\qquad
g_{\mathrm L,\lambda}(x)=\frac{\lambda}{2}\exp(-\lambda|x|),\\
g_{\mathrm T,\nu,\sigma}(x)=\frac{\Gamma\{(\nu+1)/2\}}
{\Gamma(\nu/2)(\nu\pi)^{1/2}\sigma}
\left(1+\frac{x^2}{\nu\sigma^2}\right)^{-(\nu+1)/2}.
\end{gathered}
\]

\begin{proposition}[Product slabs]
\label{prop:slab-verification}
Suppose Assumption~\ref{ass:sparse-fisher} holds and $r_n^0/\sqrt{s_n^\star\log p}\to\infty$. Put $\rho_n=r_n^0/(n c_F)^{1/2}$. Assumption~\ref{ass:slab} and the oracle slab condition in Assumption~\ref{ass:oracle-laplace} hold for each of the following coherent product families $g_S(\beta_S)=\prod_{j\in S}g(\beta_j)$:
\begin{enumerate}[label=\textnormal{(\roman*)},leftmargin=*]
\item $g=g_{\mathrm G,\sigma}$, provided
$\|\beta^0\|_\infty^2=O(\log p)$ and
$\rho_n\|\beta^0\|_2+\rho_n^2=O(1)$;
\item $g=g_{\mathrm L,\lambda}$, provided
$\|\beta^0\|_\infty=O(\log p)$ and $\sqrt{s_n^\star}\rho_n=O(1)$;
\item $g=g_{\mathrm T,\nu,\sigma}$, provided
$\log(1+\|\beta^0\|_\infty)=O(\log p)$ and $\sqrt{s_n^\star}\rho_n=O(1)$.
\end{enumerate}
In each case the slab has bounded log-density variation on the selection neighbourhoods and vanishing log-density variation on every fixed multiple of the oracle scale $\sqrt{s_0/n}$.
\end{proposition}

The Gaussian condition couples the truth's Euclidean norm with the local likelihood radius. The Laplace and Student-$t$ log-densities are globally Lipschitz, so their flatness conditions do not depend on the truth size. The Student-$t$ slab has only logarithmic point-density cost, permitting coefficients as large as a fixed power of $p$.

\subsection{Examples of generalized linear models}
\label{sec:glm-examples}

Table~\ref{tab:six-models} collects the model-specific formulas and restrictions. A common average-spread condition covers the five non-Gaussian likelihoods; Gaussian regression is exactly quadratic and does not require it.

\begin{table}[!ht]
\centering
\setlength{\tabcolsep}{6pt}
\tbl{Natural-parameter maps, Fisher weights and model-specific restrictions.}{%
\footnotesize
\begin{tabular}{@{}llll@{}}
\toprule
Model & $\xi(\eta)$ & Fisher weight $w_i^0$ & restriction\\
\midrule
Gaussian & $\eta$ & $\tau_i^{-1}$ & known bounded variances\\
Logistic & $\eta$ & $e^{\eta_i^0}(1+e^{\eta_i^0})^{-2}$ & none\\
Poisson & $\eta$ & $e^{\eta_i^0}$ & none\\
Probit & $\log\{\Phi(\eta)/(1-\Phi(\eta))\}$
& $\phi(\eta_i^0)^2/[\Phi(\eta_i^0)\{1-\Phi(\eta_i^0)\}]$
& $\max_i|\eta_i^0|\le L$\\
Gamma log-link & $-e^{-\eta}$ & $\tau_i^{-1}$ & known bounded dispersions\\
Negative-binomial log-link
& $\eta-\log(\kappa+e^\eta)$
& $\kappa e^{\eta_i^0}/(\kappa+e^{\eta_i^0})$ & fixed known $\kappa$\\
\bottomrule
\end{tabular}}
\label{tab:six-models}

\vspace{1pt}
\parbox{\dimexpr\textwidth-2\tabledim\relax}{\scriptsize Variances and dispersions are bounded above and away from zero; $L<\infty$ and $\kappa>0$ are fixed. ``None'' means no extra model-specific restriction beyond the common assumptions. Here $\phi$ and $\Phi$ are the standard normal density and distribution functions.}
\end{table}

For $S\in\cS_n^+$, put $z_{i,S}=(\mF_S^0)^{-1/2}X_{i,S}$. We use the entry bounds
\begin{align}
\max_{i,j}|X_{ij}|&\le K_X,
\qquad
\max_{i,j}(w_i^0)^{1/2}|X_{ij}|\le K_V,
\label{eq:simple-design-entries}
\end{align}
for fixed $K_X,K_V<\infty$, where $w_i^0$ is the truth-Fisher weight. For the five non-Gaussian models, average spread means that, uniformly over $S\in\cS_n^+$ and unit $u$,
\begin{equation}
 \sum_iw_i^0|z_{i,S}^\top u|^3\le Cn^{-1/2},
 \qquad
 \sum_i(w_i^0)^2(z_{i,S}^\top u)^4\le Cn^{-1}.
 \label{eq:simple-spread-moments}
\end{equation}

\begin{proposition}[Fixed-design GLMs]
\label{prop:simple-fixed-design}
Consider any of the six models in Table~\ref{tab:six-models} under its stated restriction. Suppose Assumption~\ref{ass:sparse-fisher} and \eqref{eq:simple-design-entries} hold. For a non-Gaussian model, also suppose \eqref{eq:simple-spread-moments}. Then Assumptions~\ref{ass:score} and \ref{ass:truth-lan} and the likelihood part of Assumption~\ref{ass:oracle-laplace} hold under the applicable rates below, with only the first required for Gaussian regression,
\begin{equation}
 \frac{s_0^2\log p}{n}\longrightarrow0,
 \qquad
 \begin{cases}
 s_0^3/(n\log p)\longrightarrow0,
 &\text{logistic, Poisson and bounded-window probit},\\
 s_0^3/n\longrightarrow0,
 &\text{Gamma and negative binomial}.
 \end{cases}
 \label{eq:simple-fixed-growth}
\end{equation}
\end{proposition}

For bounded Fisher weights, the second entry bound in \eqref{eq:simple-design-entries} follows from the first. Proposition~\ref{supp-prop:primitive-bridge} in the supplement relaxes these entry bounds, and Proposition~\ref{supp-prop:random-design-geometry} verifies average spread for sub-Gaussian designs. Subsections~\ref{supp-sec:fixed-design-verification} and~\ref{supp-sec:six-model-cards} give the score, curvature and concavity calculations for the following examples.

\begin{example}[Gaussian regression]
The score is Gaussian and the likelihood is quadratic, so neither average spread nor an additional non-Gaussian rate is needed.
\end{example}

\begin{example}[Logistic regression]
The Fisher weight is bounded by $1/4$. The bounded Bernoulli score and response-independent curvature require no truth-predictor window.
\end{example}

\begin{example}[Poisson regression]
The exact centred moment-generating function and canonical curvature give the required likelihood bounds. An upper predictor window is sufficient for the weighted-entry bound but is not otherwise required. Proposition~\ref{prop:simple-random-design} also covers unbounded Fisher weights.
\end{example}

\begin{example}[Probit regression]
The stated predictor window bounds the link derivatives and Fisher weights. Together with bounded Bernoulli residuals, this gives the required score and curvature controls.
\end{example}

\begin{example}[Gamma regression]
For the log link, standardization by the truth mean gives uniform score and curvature bounds without a truth-predictor window.
\end{example}

\begin{example}[Negative-binomial regression]
The Fisher weight is bounded by $\kappa$. The link derivative offsets the response-scale growth, giving uniform score and curvature bounds without a truth-predictor window.
\end{example}

In particular, to obtain the logistic benchmark in Table~\ref{tab:closest-work}, impose the sparse Fisher geometry and the fixed-design conditions in Proposition~\ref{prop:simple-fixed-design}. Use $\pi_p(s)\propto p^{-as}$ with fixed $a>1$, a fixed-scale Laplace product slab and $\|\beta^0\|_\infty=O(\log p)$. For sufficiently large $K_0$ and $C_\alpha$, the conditions
\begin{equation}\label{eq:logistic-benchmark-growth}
s_0\le\log p,\qquad s_0^2\log p=o(n),\qquad
nb_0^2\ge C_\alpha\log p
\end{equation}
give exact selection and the oracle BvM theorem. The constant $C_\alpha$ may depend on $a$. The true-Fisher covariance replacement and marginal coverage conditions for every fixed active coordinate are automatic, as are the ellipsoid coverage conditions for fixed $s_0$. Under $s_0\le\log p$, the growth condition $s_0^2\log p=o(n)$ also implies the secondary logistic rate $s_0^3=o(n\log p)$ in Proposition~\ref{prop:simple-fixed-design}. The oracle curvature requirement is weaker. Since $\bar\kappa_n^0=O(n^{-1/2})$, the condition $s_0\bar\kappa_n^0=o(1)$ requires only $s_0^2=o(n)$. The conditional Gaussian approximation also uses localization and slab flatness, as specified in Proposition~\ref{prop:conditional-bvm}. The radius choices and these refinements are verified in the supplement.

\subsection{Random designs}
\label{sec:random-design-examples}

The random-design result derives the fixed-design likelihood conditions from sub-Gaussian tails and population sparse-eigenvalue bounds. Write $\mathbb P_{\mX}$ and $\mathbb E_{\mX}$ for probability and expectation under the design law, and let $\|Z\|_{\psi_2}=\inf\{c>0:\mathbb E_{\mX}\exp(Z^2/c^2)\le2\}$.

The bounded-weight route covers Gaussian, logistic, probit, Gamma log-link and fixed-size negative-binomial regression under the corresponding restrictions above. The probit predictor window must hold $\mathbb P_{\mX}$-almost surely, which excludes nondegenerate Gaussian designs with $\beta^0\ne0$. This route also covers Poisson regression when $\max_iX_i^\top\beta^0\le L<\infty$ almost surely. Set $a_i=w_i^0$ for this route. The unbounded-weight route is specific to Poisson regression and assumes $\|\beta^0\|_2\le B<\infty$. Set $a_i=1$ for this route. In both cases define $\overline\Sigma_n=(1/n)\sum_{i=1}^n\mathbb E_{\mX}(a_iX_iX_i^\top)$.

\begin{proposition}[Sub-Gaussian designs]
\label{prop:simple-random-design}
Suppose $X_1,\ldots,X_n$ are independent and $\sup_i\sup_{\|u\|_2=1}\|u^\top X_i\|_{\psi_2}\le K$ for a fixed $K<\infty$. Suppose one of the two routes above applies and the eigenvalues of $\overline\Sigma_n$ on every coordinate subspace of size at most $s_n^\star$ are bounded above and away from zero by fixed constants. Suppose also that $s_n^\star=o(p)$ and $\log n=O(\log p)$. Then there are design events $\mathcal X_n$, with $\mathbb P_{\mX}(\mathcal X_n)\to1$, such that Assumptions~\ref{ass:sparse-fisher} and \ref{ass:score} hold for every $\mX\in\mathcal X_n$ and the likelihood events in Assumptions~\ref{ass:truth-lan} and \ref{ass:oracle-laplace} have conditional probability $1-o(1)$ uniformly over $\mathcal X_n$, provided
\begin{equation}\label{eq:simple-random-growth}
s_0\le\log p,\qquad s_0^2(\log p)^2/n\longrightarrow0.
\end{equation}
\end{proposition}

Proposition~\ref{supp-prop:random-design-geometry} in the supplement proves \eqref{eq:simple-spread-moments} for the bounded-weight route, while Proposition~\ref{supp-prop:poisson-subgaussian-bridge} gives the Poisson result with unbounded Fisher weights. The separate Poisson route is needed because $(w_i^0)^{1/2}X_i$ need not be sub-Gaussian when the Fisher weights are unbounded. For fixed $s_0$, \eqref{eq:simple-random-growth} reduces to $(\log p)^2=o(n)$. The supplement retains the general rates without $s_0\le\log p$.

\begin{corollary}[Oracle theory for six GLMs]
\label{cor:primitive-six-models}
Fix $\alpha\in(0,1]$. Suppose the likelihood conditions are verified by Proposition~\ref{prop:simple-fixed-design} or Proposition~\ref{prop:simple-random-design}. Suppose Proposition~\ref{prop:slab-verification} applies to the chosen slab at the same likelihood radius, and Assumptions~\ref{ass:prior} and \ref{ass:signal} hold. Then Theorems~\ref{thm:selection}--\ref{thm:oracle-bvm} and Corollaries~\ref{cor:posterior-contraction}--\ref{cor:selected-coverage} hold. The frequentist-coverage conclusions remain subject to the sampling-calibration conditions in Corollary~\ref{cor:coverage}.
\end{corollary}

\section{Discussion}

The theory is developed under correct specification, known variance or dispersion parameters, fixed known negative-binomial size, samplewise concavity and a nonempty true support separated from zero. Allowing for misspecification or unknown nuisance parameters would require calibration and joint posterior control \citep{kleijn2012misspecification,bickel2012semiparametric}. Weak signals and the null model call for inferential targets that do not rely on exact support recovery. The coverage conclusions are pointwise under the stated signal separation. Uniform inference near model boundaries remains outside the present theory \citep{leeb2008sparse}.

The selected-support conclusions concern the ideal posterior modal support. Each support probability is mathematically well defined, but its integral generally has no closed form and maximization over supports can be infeasible for large $p$ \citep{tang2024empirical}. Laplace approximations and stochastic search may provide computational surrogates, but posterior concentration does not ensure that they locate the dominant support or mix rapidly \citep{yang2016computational}. Error bounds sufficient to preserve the coverage conclusions remain to be established.

\section*{Acknowledgement}
The authors thank Ray Bai for helpful discussions and suggestions on the oracle properties of Bayesian sparse estimators. Lu gratefully acknowledges support from a Discovery Grant (RGPIN-2024-03951) awarded by the Natural Sciences and Engineering Research Council of Canada.

\section*{Supplementary material}
The Supplementary Material includes all proofs and technical results.

% A single bibliography serves the main text and supplementary proofs.
\bibliographystyle{abbrvnat}
\bibliography{references}

% One combined document: all proof-to-main references are internal.
\clearpage
\appendix
\setcounter{section}{0}
\setcounter{equation}{0}
\setcounter{table}{0}
\setcounter{figure}{0}
\setcounter{theorem}{0}
\setcounter{lemma}{0}
\setcounter{corollary}{0}
\setcounter{proposition}{0}
\setcounter{assumption}{0}
\setcounter{remark}{0}
\renewcommand{\thesection}{\Alph{section}}
\numberwithin{equation}{section}
\numberwithin{table}{section}
\renewcommand{\theequation}{\thesection.\arabic{equation}}
\renewcommand{\thelemma}{\thesection.\arabic{lemma}}
\renewcommand{\thecorollary}{\thesection.\arabic{corollary}}
\renewcommand{\thetheorem}{\thesection.\arabic{theorem}}
\renewcommand{\theproposition}{\thesection.\arabic{proposition}}
\renewcommand{\theassumption}{\thesection.\arabic{assumption}}
\renewcommand{\theremark}{\thesection.\arabic{remark}}
\makeatletter
\@addtoreset{theorem}{section}
\let\c@lemma\c@theorem
\let\c@corollary\c@theorem
\let\c@proposition\c@theorem
\let\c@assumption\c@theorem
\let\c@remark\c@theorem
\makeatother
% Distinct destinations prevent collisions with main-text numbering.
\renewcommand{\theHsection}{supp.\Alph{section}}
\renewcommand{\theHsubsection}{\theHsection.\arabic{subsection}}
\renewcommand{\theHequation}{\theHsection.\arabic{equation}}
\renewcommand{\theHtable}{\theHsection.\arabic{table}}
\renewcommand{\theHtheorem}{\theHsection.\arabic{theorem}}
\renewcommand{\theHlemma}{\theHsection.\arabic{theorem}}
\renewcommand{\theHcorollary}{\theHsection.\arabic{theorem}}
\renewcommand{\theHproposition}{\theHsection.\arabic{theorem}}
\renewcommand{\theHassumption}{\theHsection.\arabic{theorem}}
\renewcommand{\theHremark}{\theHsection.\arabic{theorem}}
\RenewDocumentCommand{\paragraph}{s m}{\par\smallskip\noindent\textit{#2.}\quad}

\pdfbookmark[0]{Supplementary material}{supplement}
\begin{center}
{\large\bfseries Supplementary material}\\[3pt]
{\itshape Bernstein--von Mises theorem for sparse generalized linear models}
\end{center}
\noindent\textbf{Contents}
\begin{list}{}{\setlength{\leftmargin}{0pt}\setlength{\itemsep}{1pt}
\setlength{\parsep}{0pt}\setlength{\topsep}{4pt}}
\item \hyperref[supp-sec:reader-guide]{\ref*{supp-sec:reader-guide}.\quad Roadmap and notation}\hfill\pageref{supp-sec:reader-guide}
\item \hyperref[supp-sec:selection-proof]{\ref*{supp-sec:selection-proof}.\quad Proof of selection consistency}\hfill\pageref{supp-sec:selection-proof}
\item \hyperref[supp-sec:oracle-proof]{\ref*{supp-sec:oracle-proof}.\quad Oracle approximation and inferential consequences}\hfill\pageref{supp-sec:oracle-proof}
\item \hyperref[supp-sec:model-verification]{\ref*{supp-sec:model-verification}.\quad Verification of assumptions}\hfill\pageref{supp-sec:model-verification}
\end{list}

\section{Roadmap and notation}
\label{supp-sec:reader-guide}

Section~\ref{supp-sec:selection-proof} proves selection consistency. Section~\ref{supp-sec:oracle-proof} establishes the conditional Gaussian approximation, combines it with selection to obtain the oracle BvM theorem and contraction, and proves the credible-set and coverage results. Section~\ref{supp-sec:model-verification} verifies the assumptions for the stated product slabs, models and fixed or random designs.

Unless displayed otherwise, constants denoted by $C$ or $c$ may depend on the fixed power $\alpha$, the model-size prior constants $(a_-,a_+,c_-,c_+,K_0)$, the slab family and scale, the fixed link and likelihood envelopes, and the Fisher, score and response-Bernstein constants. They never depend on $n,p$ or $s_0$. In particular, the beta-min constant $C_\alpha$ may depend on $a_-$ and these fixed quantities. For a geometric complexity prior it may therefore depend on its exponent $a$. No constant is asserted to be uniform in $\alpha$.

We use the powered marginal likelihood $\mathcal Z_{n,\alpha}(S)$ from Section~2 and the standardized score $Z_S$ from Remark~\ref{rem:marginal-likelihoods}. Write $\omega_p(S)=\pi_p(|S|)/\binom p{|S|}$, so posterior support probabilities are proportional to $\omega_p(S)\mathcal Z_{n,\alpha}(S)$. For the null support, vectors and matrices are zero-dimensional, while the empty determinant and Gaussian integral equal one.
For the proof, define
\[
\lambda_{g,n}=\max_{j\le p}[-\log g_j(\beta_j^0)]_+,
\qquad
b_{g,n}=\sup_{S\in\cS_n^+}\sup_{\|u\|_2\le r_n^0/\sqrt{nc_F}}
\left|\log\frac{g_S(\beta_S^0+u)}{g_S(\beta_S^0)}\right|.
\]
Assumption~\ref{ass:slab} gives $\lambda_{g,n}=O(\log p)$ and $b_{g,n}=O(1)$. Put
$\underline\kappa_n=1-\varepsilon_n$ and $\overline\kappa_n=1+\varepsilon_n$.
The retained radius condition gives $s_n^\star\log p=o((r_n^0)^2)$. No upper bound on $r_n^0/\sqrt n$ is used in the abstract proofs.

\begin{lemma}[Local prior consequences]
\label{supp-le:main-prior-consequences}
Under Assumptions~\ref{ass:sparse-fisher}, \ref{ass:truth-lan} and \ref{ass:slab}, \textnormal{(ii)}--\textnormal{(iii)} hold. If Assumption~\ref{ass:prior} also holds, \textnormal{(i)} holds.
\begin{enumerate}[label=\textnormal{(\roman*)},leftmargin=*]
\item For a fixed $c_{\mathrm{sb}}>0$,
\[
\omega_p(S_0)\int_{\{\|(\mF_{S_0}^0)^{1/2}(\beta_{S_0}-\beta_{S_0}^0)\|_2
\le c_{\mathrm{sb}}\sqrt{s_0\log p}\}}
g_{S_0}(\beta_{S_0})\dif\beta_{S_0}
\ge \e^{-C_\Pi s_0\log p}.
\]
\item For every underfitted $S$ with $|S|\le s_n^\dagger$, put $M=S_0\setminus S$ and $T=S\cup S_0$. Define $L_{n,S}=\exp(b_{g,n})\prod_{j\in M}g_j(\beta_j^0)^{-1}$. Then
\[
\sup_{\substack{u\in\mathbb R^{|S|}:\\(u,-\beta_M^0)\in\cB_T(r_n^0)}}
\frac{g_S(\beta_S^0+u)}{g_T(\beta_T^0)}\le L_{n,S}.
\]
\item For every fixed $c>0$ used below, uniformly over $S\in\cS_n^+$,
\[
\frac{\e^{-c(r_n^0)^2}}
{g_S(\beta_S^0)(2\pi)^{|S|/2}|\mF_S^0|^{-1/2}}=o(1).
\]
\end{enumerate}
\end{lemma}

\begin{proof}
For \textnormal{(i)}, the model-size recursions in Assumption~\ref{ass:prior} give $\pi_p(0)\ge1/2$ eventually and
$\pi_p(s_0)\ge\tfrac12c_-^{s_0}p^{-a_-s_0}$.
After the change of variables $h=(\mF_{S_0}^0)^{1/2}(\beta_{S_0}-\beta_{S_0}^0)$, bounded slab variation bounds the integral below by the truth density times the Fisher determinant factor and the Euclidean-ball volume.
Uniform support weights, $\lambda_{g,n}=O(\log p)$, the Fisher bounds and Stirling's formula give a logarithmic lower bound $-Cs_0\log p$ for this product, proving \textnormal{(i)}.
For \textnormal{(ii)}, sparse Fisher geometry gives $\|u\|_2\le r_n^0/\sqrt{nc_F}$. Apply bounded slab variation on $T$ to $(u,0_M)$ and use product coherence to obtain the stated bound.
For \textnormal{(iii)}, the Fisher and truth-density bounds in Assumptions~\ref{ass:sparse-fisher} and \ref{ass:slab} give a lower bound $\exp(-Cs_n^\star\log p)$ for the denominator.
Since $s_n^\star\log p=o((r_n^0)^2)$, \textnormal{(iii)} follows.
\end{proof}

For any partition $T=S\dot\cup M$ with $T\in\cS_n^+$, partition $\mF_T^0$ and $\Delta_T$ accordingly and define the Schur complement by $\mK_{M\mid S}=\mF_M^0-\mF_{MS}^0(\mF_S^0)^{-1}\mF_{SM}^0$, projected score by $\Delta_{M\mid S}=\Delta_M-\mF_{MS}^0(\mF_S^0)^{-1}\Delta_S$, and normalized projected score by $Z_{M\mid S}=\mK_{M\mid S}^{-1/2}\Delta_{M\mid S}$. For an underfitted $S$ with $|S|\le s_n^\dagger$, take $M=S_0\setminus S$ and $T=S\cup S_0$, and put $d_S^2=(\beta_M^0)^\top\mK_{M\mid S}\beta_M^0$. The affine comparison factor is
\begin{align}
\mathfrak A_{n,\alpha}(S)
&=\frac{\omega_p(S)}{\omega_p(T)}L_{n,S}
\left(\frac{\alpha}{2\pi}\right)^{|M|/2}
|\mK_{M\mid S}|^{1/2}.
\label{supp-eq:rec-A}
\end{align}
For $d\ge0$, put $\psiEff(d)=d^2\wedge(d/d_n)$, with $d/d_n=+\infty$ when $d_n=0$.

\section{Proof of selection consistency}
\label{supp-sec:selection-proof}

\subsection{Uniform likelihood toolkit}
\label{supp-sec:likelihood-toolkit}

\begin{lemma}[Uniform score bounds]\label{supp-le:rec-score}
Under Assumptions~\ref{ass:sparse-fisher}--\ref{ass:truth-lan}, there are $K_{\mathrm{sc}},K_{\mathrm{sc}}^0<\infty$ such that, with $\mP_{\beta^0}$-probability tending to one,
\begin{align}
\sup_{T\in\cS_n^+}\|Z_T\|_2
&\le K_{\mathrm{sc}}\sqrt{s_n^\star\log p},
\label{supp-eq:rec-uniform-score}\\
\|Z_{S_0}\|_2^2&\le K_{\mathrm{sc}}^0(s_0+\log p).
\label{supp-eq:rec-truth-score}
\end{align}
In particular, the left side of \eqref{supp-eq:rec-uniform-score} is at most $r_n^0/4$ with $\mP_{\beta^0}$-probability tending to one.
\end{lemma}

\begin{proof}
Assumption~\ref{ass:score} and Chernoff's inequality, with the admissible choice $\lambda=t/(1+d_nt)$, give, for every unit $u$ and $t>0$,
\begin{equation}
\mP_{\beta^0}\{|u^\top Z_T|>t\}
\le2\exp\left\{-\frac{t^2}{2(1+d_nt)}\right\}.
\label{supp-eq:rec-score-tail}
\end{equation}
A $1/2$-net of $\mathbb S^{|T|-1}$ has at most $5^{|T|}$ points and controls $\|Z_T\|_2$ by twice its largest absolute projection. The total number of net points over $\cS_n^+$ is $\exp\{O(s_n^\star\log p)\}$. Since $d_n\sqrt{s_n^\star\log p}=o(1)$, a union bound proves \eqref{supp-eq:rec-uniform-score}. The same argument on $S_0$ alone, with threshold $C\sqrt{s_0+\log p}$, proves \eqref{supp-eq:rec-truth-score} with $K_{\mathrm{sc}}^0$ independent of $K_0$. The bound by $r_n^0/4$ follows from $r_n^0/\sqrt{s_n^\star\log p}\to\infty$.

\end{proof}

\begin{lemma}[Projected-score moments]\label{supp-le:projected-moments}
Under Assumptions~\ref{ass:sparse-fisher} and \ref{ass:score}, every unit projection of $Z_{M\mid S}$, for $T=S\dot\cup M\in\cS_n^+$, satisfies the Bernstein bound in Assumption~\ref{ass:score}.
For any fixed $t<1/2$ with $t>0$, and $T_n=C\sqrt{s_n^\star\log p}$ with fixed $C<\infty$,
\begin{equation}
\mE_{\beta^0}\left[
\e^{t\|Z_{M\mid S}\|_2^2}
\mathbf 1_{\{\|Z_{M\mid S}\|_2\le T_n\}}
\right]\le C_t^{|M|},
\label{supp-eq:projected-moment}
\end{equation}
uniformly over these partitions, for all large $n$.
\end{lemma}

\begin{proof}
For a unit $u\in\mathbb R^{|M|}$, let
\[
L=\begin{pmatrix}-(\mF_S^0)^{-1}\mF_{SM}^0\\ \mI_{|M|}\end{pmatrix},
\qquad
v=(\mF_T^0)^{1/2}L\mK_{M\mid S}^{-1/2}u.
\]
The identities $L^\top\Delta_T=\Delta_{M\mid S}$ and $L^\top\mF_T^0L=\mK_{M\mid S}$ give
\begin{equation}
u^\top Z_{M\mid S}=v^\top Z_T,\qquad \|v\|_2=1.
\label{supp-eq:rec-schur-score}
\end{equation}
Thus Assumption~\ref{ass:score} and its one-sided and two-sided Chernoff bounds apply to the projected score.
Write $r=|M|\ge1$. A $\delta$-net of its unit sphere gives, for $0\le x\le T_n$,
\[
\mP_{\beta^0}(\|Z_{M\mid S}\|_2>x)
\le2(1+2/\delta)^r
\exp\left\{-\frac{(1-\delta)^2x^2}{2(1+d_nx)}\right\}.
\]
Choose fixed $\delta>0$ sufficiently small. Since $d_nT_n=o(1)$, the coefficient of $x^2$ in this exponent exceeds $t$ by a fixed positive amount for all large $n$. Integrating the bound against $2tx\e^{tx^2}$ on $[0,T_n]$ proves \eqref{supp-eq:projected-moment}, after increasing $C_t$. The case $r=0$ is immediate.
\end{proof}

Let $\mathcal G_n$ be the intersection of the uniform-score and true-score events in Lemma~\ref{supp-le:rec-score} and the information event in Assumption~\ref{ass:truth-lan}. Then $\mP_{\beta^0}(\mathcal G_n)\to1$. On this event,
$\|Z_{B\mid S_0}\|_2\le\|Z_{S_0\cup B}\|_2\le T_n$
with $T_n=K_{\mathrm{sc}}\sqrt{s_n^\star\log p}$.
For $B\subseteq S_0^c$ with $T=S_0\cup B\in\cS_n^+$, put
\begin{align}
\mathfrak E_{n,\alpha}
&=\left(\frac{\overline\kappa_n}{\underline\kappa_n}\right)^{s_0/2}
\exp\left\{\frac{\alpha K_{\mathrm{sc}}^0\varepsilon_n(s_0+\log p)}
{\underline\kappa_n\overline\kappa_n}\right\},\label{supp-eq:rec-E}\\
\mathfrak D_{n,\alpha}(B)
&=\frac{\omega_p(T)}{\omega_p(S_0)}g_B(0)
\left(\frac{2\pi}{\alpha\underline\kappa_n}\right)^{|B|/2}
|\mK_{B\mid S_0}|^{-1/2},\label{supp-eq:addition-weight}\\
\mathfrak C_{n,\alpha}(T)
&=\mathfrak E_{n,\alpha}\mathfrak D_{n,\alpha}(B)
\exp\left\{\frac{\alpha}{2\underline\kappa_n}
\|Z_{B\mid S_0}\|_2^2\right\}.
\label{supp-eq:rec-C}
\end{align}
Here $\mathfrak D_{n,\alpha}(\varnothing)=1$ and $\mathfrak C_{n,\alpha}(S_0)=\mathfrak E_{n,\alpha}$.
The factor $\mathfrak D_{n,\alpha}(B)$ is deterministic, whereas $\mathfrak C_{n,\alpha}(T)$ retains the random added-coordinate score. Define
\begin{equation}
\mathfrak O_{n,\alpha}=\sum_{\substack{U\supsetneq S_0\\|U|\le s_n^\star}}
\mathfrak C_{n,\alpha}(U),\qquad \eta=3/4.
\label{supp-eq:overfitting-sum}
\end{equation}

\begin{lemma}[Likelihood localization]\label{supp-le:rec-lan}
Under Assumptions~\ref{ass:sparse-fisher}--\ref{ass:truth-lan}, on the intersection of the uniform-score event in Lemma~\ref{supp-le:rec-score} and the information event in Assumption~\ref{ass:truth-lan}, uniformly over $T\in\cS_n^+$ and $\delta_T\in\cB_T(r_n^0)$,
\begin{align}
\Delta_T^{\top}\delta_T-\frac{\overline\kappa_n}{2}
\delta_T^{\top}\mF_T^0\delta_T
&\le
\ell_{n,T}(\beta_T^0+\delta_T)-\ell_n(\beta^0)\notag\\
&\le
\Delta_T^{\top}\delta_T-\frac{\underline\kappa_n}{2}
\delta_T^{\top}\mF_T^0\delta_T.
\label{supp-eq:rec-sandwich}
\end{align}
Moreover, there are fixed $c_{\mathrm{rad}},c_1>0$ such that, if $r=\|(\mF_T^0)^{1/2}\delta_T\|_2>r_n^0$, then
\begin{equation}
\ell_{n,T}(\beta_T^0+\delta_T)-\ell_n(\beta^0)
\le-c_{\mathrm{rad}}(r_n^0)^2-c_1r_n^0(r-r_n^0).
\label{supp-eq:rec-radial}
\end{equation}
Consequently, for every fixed $\alpha\in(0,1]$, any support density $g_S$ and any affine subset of a truth-containing support,
\begin{equation}
\int_{r>r_n^0}\e^{\alpha\ell_n(\beta^0+\delta)}g_S(\beta_S)\dif\beta_S
\le \e^{\alpha\ell_n(\beta^0)-\alpha c_{\mathrm{rad}}(r_n^0)^2}.
\label{supp-eq:rec-outside}
\end{equation}
\end{lemma}

\begin{proof}
Taylor's formula with integral remainder and Assumption~\ref{ass:truth-lan} give \eqref{supp-eq:rec-sandwich} under the stated relative-information error $\varepsilon_n=o(1)$. For the radial statement, fix an $\mF_T^0$-unit vector $v$ and set $g_v(t)=\ell_{n,T}(\beta_T^0+tv)$. At $t=r_n^0$, the upper quadratic sandwich and the uniform-score event give
\[
g_v(r_n^0)-g_v(0)
\le r_n^0 g_v'(0)-\frac{\underline\kappa_n}{2}(r_n^0)^2
\le-\frac18(r_n^0)^2
\]
eventually. For a concave function, successive secant slopes are nonincreasing. Thus, for $t>r_n^0$,
\[
\frac{g_v(t)-g_v(r_n^0)}{t-r_n^0}
\le\frac{g_v(r_n^0)-g_v(0)}{r_n^0}
\le-\frac18r_n^0,
\]
which gives \eqref{supp-eq:rec-radial} with fixed positive $c_{\mathrm{rad}},c_1$. Exponentiating and integrating against a probability density proves \eqref{supp-eq:rec-outside}.
\end{proof}

\begin{lemma}[Oracle maximum-likelihood localization]
\label{supp-le:oracle-mle}
Under Assumptions~\ref{ass:sparse-fisher}--\ref{ass:truth-lan}, the oracle maximum-likelihood estimator exists uniquely with probability tending to one and satisfies
\[
\| (\mF_{S_0}^0)^{1/2}(\widehat\beta_{S_0}-\beta_{S_0}^0)\|_2
=O_{\mP_{\beta^0}}(\sqrt{s_0}).
\]
Moreover, $\widehat J_{S_0}\succ0$ on an event of probability tending to one. Define
\[
\begin{aligned}
A_n&=(\mF_{S_0}^0)^{-1/2}
\left[\int_0^1-\nabla_{S_0}^2\ell_{n,S_0}
\{\beta_{S_0}^0+t(\widehat\beta_{S_0}-\beta_{S_0}^0)\}\dif t\right]
(\mF_{S_0}^0)^{-1/2},\\
B_n&=(\mF_{S_0}^0)^{-1/2}\widehat J_{S_0}(\mF_{S_0}^0)^{-1/2}.
\end{aligned}
\]
Then, with probability tending to one,
\begin{equation}
(\mF_{S_0}^0)^{1/2}(\widehat\beta_{S_0}-\beta_{S_0}^0)=A_n^{-1}Z_{S_0},
\qquad
\|A_n-\mI_{s_0}\|_{\rm op}\vee\|B_n-\mI_{s_0}\|_{\rm op}
\le\varepsilon_n.
\label{supp-eq:rec-mle-score-linearization}
\end{equation}
\end{lemma}

\begin{proof}
On the information event in Assumption~\ref{ass:truth-lan}, the likelihood is strictly concave throughout $\beta_{S_0}^0+\cB_{S_0}(r_n^0)$. The score has $\mE_{\beta^0}\|Z_{S_0}\|_2^2=s_0$, whereas $r_n^0/\sqrt{s_0}\to\infty$. Thus, with probability tending to one, the upper quadratic bound in \eqref{supp-eq:rec-sandwich} is strictly negative relative to $\ell_n(\beta^0)$ on the boundary of $\beta_{S_0}^0+\cB_{S_0}(r_n^0/2)$. A maximizer on this compact ball therefore lies in its interior. Its score vanishes, so samplewise concavity makes it a global maximizer. Strict concavity in the ball rules out a second global maximizer: the segment joining two maximizers would contradict strict concavity near the first. The information bound gives $\widehat J_{S_0}\succ0$. Integrating the Hessian along the segment from the truth to the maximizer yields \eqref{supp-eq:rec-mle-score-linearization}. Since $\|A_n^{-1}\|_{\rm op}\le(1-\varepsilon_n)^{-1}$, the score second moment gives the asserted rate.
\end{proof}

\begin{proof}[Proof of Remark~\ref{rem:marginal-likelihoods}]
Use $h=(\mF_T^0)^{1/2}(\beta_T-\beta_T^0)$. Lemma~\ref{supp-le:rec-lan} and bounded slab variation give, uniformly on $\|h\|_2\le r_n^0$, upper and lower bounds by the kernels
\[
\e^{\alpha\ell_n(\beta^0)}
g_T(\beta_T^0)|\mF_T^0|^{-1/2}
\exp\{\alpha(h^{\top}Z_T-c\|h\|_2^2/2)\}\dif h
\]
with $c=\underline\kappa_n$ and $c=\overline\kappa_n$, respectively, up to slab factors between $\e^{-b_{g,n}}$ and $\e^{b_{g,n}}$, which are uniformly bounded above and away from zero. The corresponding full Gaussian integrals are $\mathcal J_{n,\alpha}(T,c)$. On the uniform-score event their centres $Z_T/c$ have norm at most $r_n^0/3$. Since $|T|=o((r_n^0)^2)$, their mass outside the $r_n^0$-ball is uniformly negligible. Lemma~\ref{supp-le:rec-lan} controls the original likelihood outside the ball, and Lemma~\ref{supp-le:main-prior-consequences}\textnormal{(iii)} makes that absolute tail negligible relative to either local scale.
\end{proof}

\subsection{Posterior denominator and support control}
\label{supp-sec:denominator-dimension}

\begin{samepage}
\begin{lemma}[Denominator and support control]\label{supp-le:rec-support-control}
Under Assumptions~\ref{ass:sparse-fisher}--\ref{ass:slab}, for every fixed $\alpha\in(0,1]$, the following hold.
\begin{enumerate}[label=\textnormal{(\roman*)},leftmargin=*]
\item There is $C_D<\infty$, independent of $K_0$, such that, with $\mP_{\beta^0}$-probability tending to one,
\[
\int\exp\{\alpha[\ell_n(\beta)-\ell_n(\beta^0)]\}\dif\Pi_n(\beta)
\ge\e^{-C_Ds_0\log p}.
\]
\item The posterior model dimension satisfies
\[
\mE_{\beta^0}\Pi_{n,\alpha}(|S_\beta|>s_n^\dagger\mid Y)\longrightarrow0.
\]
\item The strict-superset comparison factor satisfies
\[
\mE_{\beta^0}\{\mathfrak O_{n,\alpha}^{\eta}\mathbf 1_{\mathcal G_n}\}=o(1).
\]
\end{enumerate}
\end{lemma}
\end{samepage}

\begin{proof}
For \textnormal{(i)}, restrict the integral to the $S_0$ Fisher ball of radius $c_{\mathrm{sb}}\sqrt{s_0\log p}$. The lower sandwich and the true-support score bound \eqref{supp-eq:rec-truth-score} give $\ell_{n,S_0}(\beta_{S_0})-\ell_n(\beta^0)\ge-C s_0\log p$ there, with $C$ independent of $K_0$. Combining this with its prior mass in Lemma~\ref{supp-le:main-prior-consequences}\textnormal{(i)} proves the claim. Both bounds use only the true support, so $C_D$ depends on the fixed power, prior and Fisher bounds and $K_{\mathrm{sc}}^0$, but not on $K_0$.

For \textnormal{(ii)}, iterating the upper recursion in Assumption~\ref{ass:prior} gives $\sum_{k>s_n^\dagger}\pi_p(k) \le C\{c_+p^{-a_+}\}^{s_n^\dagger+1}$. Choose $K_0$ so that $a_+K_0>C_D$. This choice gives
\begin{equation}
\e^{C_Ds_0\log p}\sum_{k>s_n^\dagger}\pi_p(k)=o(1),
\label{supp-eq:rec-size-prior-tail}
\end{equation}
where $C_D$ is from \textnormal{(i)}. For every $\beta$, $\mE_{\beta^0}\exp\{\alpha[\ell_n(\beta)-\ell_n(\beta^0)]\}\le1$ by Jensen's inequality when $0<\alpha<1$, with equality at $\alpha=1$. Thus Fubini's theorem and the denominator bound control the expected posterior tail by the left side of \eqref{supp-eq:rec-size-prior-tail} plus the probability that the denominator event fails. This proves \textnormal{(ii)} at both ordinary and fractional powers.

For \textnormal{(iii)}, Assumption~\ref{ass:truth-lan} gives
\begin{equation}
\log\mathfrak E_{n,\alpha}
\le C\varepsilon_n(s_0+\log p)=o(\log p).
\label{supp-eq:rec-E-bound}
\end{equation}
Let $U=S_0\cup B$, $r=|B|\ge1$, $\rho_n=r_n^0/\sqrt{nc_F}$, and let $v_r$ be the volume of the unit Euclidean ball in $\mathbb R^r$. Product coherence and bounded slab variation give
\[
1=\int g_B(u)\dif u
\ge \e^{-b_{g,n}}g_B(0)v_r\rho_n^r.
\]
Sparse Fisher geometry gives $\mK_{B\mid S_0}\succeq nc_F\mI_r$, so
\begin{equation}
g_B(0)|\mK_{B\mid S_0}|^{-1/2}
\le\frac{\e^{b_{g,n}}}{v_r(r_n^0)^r}
\le\e^{b_{g,n}}\left(\frac{C\sqrt r}{r_n^0}\right)^r.
\label{supp-eq:slab-normalization}
\end{equation}
The upper model-size recursion and $s_n^\star=o(p)$ yield
\[
\frac{\omega_p(U)}{\omega_p(S_0)}
\le\{Cs_0p^{-(a_++1)}\}^r.
\]
For any fixed $C_0<\infty$, therefore,
\begin{align}
\sum_{|B|=r}C_0^r\mathfrak D_{n,\alpha}(B)^\eta
&\le
\left\{C p^{1-\eta(a_++1)}
\left(\frac{s_0}{r_n^0}\right)^\eta r^{\eta/2-1}\right\}^{r}
\notag\\
&\le\{Cp^{-\delta}r^{\eta/2-1}\}^{r},
\quad \delta=\eta(a_++1/2)-1>1/8.
\label{supp-eq:rec-addition-layer}
\end{align}
We used $\binom{p-s_0}{r}\le(ep/r)^r$ and
$s_0/r_n^0\le C\sqrt{s_0/\log p}\le C\sqrt p$ eventually. The last bound is summable over $r\ge1$ and its sum is $O(p^{-\delta})$.
On $\mathcal G_n$, the projected scores are truncated at $T_n$. Since $\eta\alpha/(2\underline\kappa_n)$ is eventually bounded by a fixed number below $1/2$, Lemma~\ref{supp-le:projected-moments} gives
\[
\mE_{\beta^0}\{\mathfrak O_{n,\alpha}^\eta\mathbf 1_{\mathcal G_n}\}
\le\mathfrak E_{n,\alpha}^\eta
\sum_{B\ne\varnothing}C_0^{|B|}\mathfrak D_{n,\alpha}(B)^\eta
=p^{o(1)}O(p^{-\delta})=o(1).
\]
Here $(\sum_jx_j)^\eta\le\sum_jx_j^\eta$ for nonnegative $x_j$. This proves \textnormal{(iii)}.
\end{proof}

\subsection{Underfitted supports}
\label{supp-sec:underfitting}

\begin{lemma}[Affine-slice bound]\label{supp-le:rec-affine}
Under Assumptions~\ref{ass:sparse-fisher}--\ref{ass:truth-lan} and \ref{ass:slab}, fix $\alpha\in(0,1]$ and $S\nsupseteq S_0$ with $|S|\le s_n^\dagger$, and put $M=S_0\setminus S$, $B=S\setminus S_0$, $T=S\cup S_0$, and $\gamma=\beta_M^0$. On $\mathcal G_n$, embed the underfitted model in $T$ through $\delta_T=(u,-\gamma)$, where $u=\beta_S-\beta_S^0$.
Let $a_S=(\mK_{M\mid S})^{1/2}\gamma$,
so $\|a_S\|_2=d_S$. For $c>0$, let $Q_{T,c}(\delta)=\Delta_T^{\top}\delta-\frac c2\delta^{\top}\mF_T^0\delta$,
so that the exact profile identity is
\begin{equation}
\sup_{\delta\in\mathbb R^{|T|}}Q_{T,c}(\delta)
-\sup_{u\in\mathbb R^{|S|}}Q_{T,c}(u,-\gamma)
=\frac1{2c}\|Z_{M\mid S}+ca_S\|_2^2.
\label{supp-eq:rec-affine-profile}
\end{equation}
Define
\begin{align*}
\mathcal Z_{n,\alpha}^{\mathrm{loc}}(S)
&=\int_{\substack{u\in\mathbb R^{|S|}:\\
\|(\mF_T^0)^{1/2}(u,-\gamma)\|_2\le r_n^0}}
\e^{\alpha\ell_{n,S}(\beta_S^0+u)}g_S(\beta_S^0+u)\dif u,\\
\mathcal Z_{n,\alpha}^{\mathrm{out}}(S)
&=\int_{\substack{u\in\mathbb R^{|S|}:\\
\|(\mF_T^0)^{1/2}(u,-\gamma)\|_2>r_n^0}}
\e^{\alpha\ell_{n,S}(\beta_S^0+u)}g_S(\beta_S^0+u)\dif u.
\end{align*}
Then $\mathcal Z_{n,\alpha}(S)=\mathcal Z_{n,\alpha}^{\mathrm{loc}}(S)+\mathcal Z_{n,\alpha}^{\mathrm{out}}(S)$. On $\mathcal G_n$, for a fixed $C<\infty$ and uniformly over the stated supports,
\begin{align}
\frac{\omega_p(S)\mathcal Z_{n,\alpha}^{\mathrm{loc}}(S)}
{\omega_p(S_0)\mathcal Z_{n,\alpha}(S_0)}
&\le C\mathfrak A_{n,\alpha}(S)
\mathfrak C_{n,\alpha}(T)\notag\\
&\quad\times
\exp\left\{-\frac{\alpha}{2\underline\kappa_n}
\|Z_{M\mid S}+\underline\kappa_na_S\|_2^2\right\},
\label{supp-eq:rec-affine-local}\\
\mathcal Z_{n,\alpha}^{\mathrm{out}}(S)
&\le
\e^{\alpha\ell_n(\beta^0)-\alpha c_{\mathrm{rad}}(r_n^0)^2}.
\label{supp-eq:rec-affine-outside}
\end{align}
\end{lemma}

\begin{proof}
Completing the square in the $S$ coordinates gives \eqref{supp-eq:rec-affine-profile}. The upper sandwich in \eqref{supp-eq:rec-sandwich} and the slab envelope bound the local integral by the full affine Gaussian integral
\[
\e^{\alpha\ell_n(\beta^0)}L_{n,S}g_T(\beta_T^0)
\left(\frac{2\pi}{\alpha\underline\kappa_n}\right)^{|S|/2}
|\mF_S^0|^{-1/2}
\exp\left\{\alpha\sup_u
Q_{T,\underline\kappa_n}(u,-\gamma)\right\}.
\]
The lower half of Remark~\ref{rem:marginal-likelihoods} lower bounds the $S_0$ integral with curvature $\overline\kappa_n$. The determinant identity $\frac{|\mF_{S_0}^0|^{1/2}}{|\mF_S^0|^{1/2}} =|\mK_{M\mid S}|^{1/2} |\mK_{B\mid S_0}|^{-1/2}$ follows by evaluating $|\mF_T^0|$ through the two partitions $T=S\dot\cup M=S_0\dot\cup B$. The corresponding profile difference is
\begin{align*}
&\sup_uQ_{T,\underline\kappa_n}(u,-\gamma)
-\sup_vQ_{S_0,\overline\kappa_n}(v)\\
&\quad=
\frac{\|Z_{B\mid S_0}\|_2^2}{2\underline\kappa_n}
+\frac{\varepsilon_n}
{\underline\kappa_n\overline\kappa_n}\|Z_{S_0}\|_2^2
-\frac{\|Z_{M\mid S}+\underline\kappa_na_S\|_2^2}
{2\underline\kappa_n}.
\end{align*}
The true-score bound controls the second positive term, while the first is retained in $\mathfrak C_{n,\alpha}(T)$. Since $|S|=s_0+|B|-|M|$ and $\underline\kappa_n\le1$, the curvature factor is at most $(\overline\kappa_n/\underline\kappa_n)^{s_0/2}\underline\kappa_n^{-|B|/2}$. Collecting the prior, bounded slab-variation and determinant factors proves \eqref{supp-eq:rec-affine-local}. Equation~\eqref{supp-eq:rec-outside} gives \eqref{supp-eq:rec-affine-outside} for every $d_S$, without any support-odds or projected-score factor.
\end{proof}

\begin{theorem}[Exclusion of underfitted models]\label{supp-th:rec-underfit}
Under Assumptions~\ref{ass:sparse-fisher}--\ref{ass:signal}, for every fixed $\alpha\in(0,1]$,
\[
\mE_{\beta^0}\Pi_{n,\alpha}(S_\beta\nsupseteq S_0\mid Y)\longrightarrow0.
\]
\end{theorem}

\begin{proof}
Lemma~\ref{supp-le:rec-support-control}\textnormal{(ii)} removes supports larger than $s_n^\dagger$. For a remaining support, write $S=(S_0\setminus M)\cup B$, $m=|M|\ge1$, $r=|B|$, and $T=S\cup S_0$.
The lower model-size recursion, lower truth-density bound and sparse Fisher geometry give
\begin{equation}
\mathfrak A_{n,\alpha}(S)
\le\e^{b_{g,n}}
\{C\sqrt n\,p^{a_-+1}\e^{\lambda_{g,n}}\}^{m}
\le p^{C_1m}
\label{supp-eq:rec-deletion-cost}
\end{equation}
for a fixed $C_1<\infty$, using $\log n=O(\log p)$.
By Lemma~\ref{supp-le:rec-affine}, the local posterior contribution is bounded on $\mathcal G_n$ by
\begin{align*}
R_S={}&C\mathfrak E_{n,\alpha}\mathfrak D_{n,\alpha}(B)p^{C_1m}
\exp\left\{\frac{\alpha}{2\underline\kappa_n}\|Z_{B\mid S_0}\|_2^2\right\}\\
&\times\exp\left\{-\frac{\alpha}{2\underline\kappa_n}
\|Z_{M\mid S}+\underline\kappa_na_S\|_2^2\right\}.
\end{align*}
For any fixed $t>0$, the event
$\|Z_{M\mid S}+\underline\kappa_na_S\|_2\le\underline\kappa_nd_S/2$
implies $(a_S/d_S)^\top Z_{M\mid S}\le-\underline\kappa_nd_S/2$. The projected-score Bernstein bound therefore gives
\begin{align*}
&\mE_{\beta^0}\exp\{-t\|Z_{M\mid S}+\underline\kappa_na_S\|_2^2\}\\
&\qquad\le
\e^{-c\psiEff(\underline\kappa_nd_S)}
+\e^{-t\underline\kappa_n^2d_S^2/4}.
\end{align*}
Sparse Fisher geometry and beta-min yield
$d_S^2\ge nc_F\|\beta_M^0\|_2^2\ge c_FC_\alpha m\log p$.
Since $\psiEff$ is nondecreasing and
$d_n\sqrt{c_FC_\alpha s_0\log p}=o(1)$ for each fixed $C_\alpha$, evaluating it at this lower bound gives
\begin{equation}
\mE_{\beta^0}\exp\{-t\|Z_{M\mid S}+\underline\kappa_na_S\|_2^2\}
\le 2p^{-c_t C_\alpha m}
\label{supp-eq:omitted-score-moment}
\end{equation}
for a fixed $c_t>0$. No upper bound on the actual signal sizes is used.

Take the $\eta=3/4$ power of $R_S$ and apply H\"older's inequality to its two score exponentials with $q=7/6$ and $q'=7$. The positive-score exponent has coefficient
$q\eta\alpha/(2\underline\kappa_n)$, eventually bounded by a fixed number below $1/2$ because $q\eta=7/8$.
Lemma~\ref{supp-le:projected-moments} applies on $\mathcal G_n$. For the negative-score factor, its coefficient $q'\eta\alpha/(2\underline\kappa_n)$ is at least the fixed value $q'\eta\alpha/2$, so \eqref{supp-eq:omitted-score-moment} applies by monotonicity. Hence
\[
\mE_{\beta^0}(R_S^\eta\mathbf 1_{\mathcal G_n})
\le C\mathfrak E_{n,\alpha}^\eta
C_0^r\mathfrak D_{n,\alpha}(B)^\eta
p^{-(cC_\alpha-\eta C_1)m}
\]
with fixed $c>0$. The two projected scores need not be independent.
The constants can be fixed before choosing $C_\alpha$ sufficiently large.
Since the local posterior mass is at most one,
\[
\mE_{\beta^0}\left[\min\left\{1,\sum_SR_S\right\}\mathbf 1_{\mathcal G_n}\right]
\le\sum_S\mE_{\beta^0}(R_S^\eta\mathbf 1_{\mathcal G_n}).
\]
Summing over $B$ by \eqref{supp-eq:rec-addition-layer}, including the empty set, bounds the right side by
\[
p^{o(1)}\sum_{m=1}^{s_0}\binom{s_0}{m}
 p^{-(cC_\alpha-\eta C_1)m}=o(1)
\]
for sufficiently large $C_\alpha$, using $\binom{s_0}{m}\le p^m$.

On the denominator event in Lemma~\ref{supp-le:rec-support-control}\textnormal{(i)}, the outside contributions are bounded by
\[
\frac{\sum_{S\nsupseteq S_0,\ |S|\le s_n^\dagger}
\omega_p(S)\mathcal Z_{n,\alpha}^{\mathrm{out}}(S)}
{\sum_U\omega_p(U)\mathcal Z_{n,\alpha}(U)}
\le\exp\{C_Ds_0\log p-\alpha c_{\mathrm{rad}}(r_n^0)^2\}=o(1).
\]
The exceptional events have probability $o(1)$ and posterior mass is bounded by one. Combining these bounds proves the theorem. Assumption~\ref{ass:oracle-laplace} is not used.
\end{proof}

\subsection{Overfitted supports and conclusion}
\label{supp-sec:overfitting}

\begin{theorem}[Exclusion of overfitted models]\label{supp-th:rec-superset}
Under Assumptions~\ref{ass:sparse-fisher}--\ref{ass:slab}, for every fixed $\alpha\in(0,1]$,
\[
\mE_{\beta^0}\Pi_{n,\alpha}(S_\beta\supsetneq S_0\mid Y)\longrightarrow0.
\]
\end{theorem}

\begin{proof}
Lemma~\ref{supp-le:rec-support-control}\textnormal{(ii)} reduces the sum to $|U|\le s_n^\dagger$.
The evidence bounds and Schur identities give, on $\mathcal G_n$,
\[
\frac{\omega_p(U)\mathcal Z_{n,\alpha}(U)}
{\omega_p(S_0)\mathcal Z_{n,\alpha}(S_0)}
\le C\mathfrak C_{n,\alpha}(U)
\]
for a fixed $C$. Thus the remaining posterior mass is at most
$\min(1,C\mathfrak O_{n,\alpha})\le C^\eta\mathfrak O_{n,\alpha}^\eta$.
Lemma~\ref{supp-le:rec-support-control}\textnormal{(iii)} and
$\mP_{\beta^0}(\mathcal G_n^c)=o(1)$ prove the claim.
\end{proof}

\begin{proof}[Proof of Theorem~\ref{thm:selection}]
Theorems~\ref{supp-th:rec-underfit} and \ref{supp-th:rec-superset} show that, under Assumptions~\ref{ass:sparse-fisher}--\ref{ass:signal} and for every fixed $\alpha\in(0,1]$, $\mE_{\beta^0}\Pi_{n,\alpha}(S_\beta\ne S_0\mid Y)\longrightarrow0$. Indeed, the complement of $\{S_\beta=S_0\}$ is the disjoint union of underfitted supports and strict supersets. Markov's inequality therefore gives $\Pi_{n,\alpha}(S_\beta=S_0\mid Y)\to1$ in $\mP_{\beta^0}$-probability, as asserted.
\end{proof}

\section{Oracle approximation and inferential consequences}
\label{supp-sec:oracle-proof}

\subsection{Conditional oracle Gaussian approximation}

The following bound is a specialization of Theorem~2.2 of \citet{katsevich2025improved} to the powered log likelihood in the coordinates used here.

\begin{lemma}[Likelihood Laplace bound]
\label{supp-le:rec-sharp-laplace}
Let $d\ge1$, let $\ell:\mathbb R^d\to\mathbb R$ be twice continuously differentiable and concave, and suppose that it has a unique maximizer $\widehat\theta$ whose negative Hessian $J=-\nabla^2\ell(\widehat\theta)$ is positive definite. For $R>0$, put
\[
\kappa(R)=
\sup_{0<\|J^{1/2}u\|_2\le R}
\frac{\|J^{-1/2}\{-\nabla^2\ell(\widehat\theta+u)-J\}
J^{-1/2}\|_{\rm op}}
{\|J^{1/2}u\|_2}.
\]
Fix $\alpha\in(0,1]$. If $R\ge6\sqrt{d/\alpha}$ and $R\kappa(R)\le1/2$, then the law $Q_\alpha(\mathrm d\theta)\propto\exp\{\alpha\ell(\theta)\}\dif\theta$ is well defined and
\begin{equation}
\left\|Q_\alpha-
\mathcal N_d\{\widehat\theta,(\alpha J)^{-1}\}\right\|_{\TV}
\le \frac{d\kappa(R)}{\sqrt{2\alpha}}
+3\exp(-\alpha R^2/9).
\label{supp-eq:supp-sharp-laplace}
\end{equation}
\end{lemma}

\begin{proof}
To verify the scaling, set $f=-\alpha\ell/n$ and $H=\nabla^2f(\widehat\theta)=\alpha J/n$. The ball in that theorem with $r=R\sqrt{\alpha/d}$ is exactly $\{\|J^{1/2}(\theta-\widehat\theta)\|_2\le R\}$. Its standardized Hessian Lipschitz coefficient is $\delta_3(r)=\sqrt{n/\alpha}\,\kappa(R)$. Hence the theorem's local condition reduces to $R\kappa(R)\le1/2$, its leading error is $d\kappa(R)/\sqrt{2\alpha}$, and its tail term is $3\exp(-\alpha R^2/9)$.
\end{proof}

\begin{proof}[Proof of Proposition~\ref{prop:conditional-bvm}]
Intersect the event in Lemma~\ref{supp-le:oracle-mle}, the score and information events used in Lemma~\ref{supp-le:rec-lan}, and the curvature event in Assumption~\ref{ass:oracle-laplace}. Their intersection has probability tending to one. Put
\[
R_n=\min\{r_n,(2\bar\kappa_n^0)^{-1}\},
\]
with $(2\bar\kappa_n^0)^{-1}=\infty$ if $\bar\kappa_n^0=0$. The retained rate conditions in Assumption~\ref{ass:oracle-laplace} give
\[
R_n/\sqrt{s_0}\longrightarrow\infty,
\qquad R_n\kappa(R_n)\le\tfrac12,
\qquad s_0\kappa(R_n)\le s_0\bar\kappa_n^0\longrightarrow0.
\]
For the first assertion, note that $\sqrt{s_0}\bar\kappa_n^0\le s_0\bar\kappa_n^0\to0$. Here $\kappa(R_n)$ is the oracle curvature modulus restricted to radius $R_n$. The curvature inequalities follow by restricting the supremum from radius $r_n$ to $R_n$. Hence Lemma~\ref{supp-le:rec-sharp-laplace} applies for all large $n$. Let $Q_{n,\alpha}$ be the probability law obtained by normalizing $\exp\{\alpha\ell_{n,S_0}(\beta_{S_0})\}$, and write
\[
G_{n,\alpha}=\mathcal N_{s_0}\{\widehat\beta_{S_0},(\alpha\widehat J_{S_0})^{-1}\}.
\]
Extend these laws arbitrarily outside the stated event. Then
\begin{equation}
\|Q_{n,\alpha}-G_{n,\alpha}\|_{\TV}=o_{\mP_{\beta^0}}(1).
\label{supp-eq:supp-likelihood-gaussian}
\end{equation}
The conditional posterior $P_{n,\alpha}=\Pi_{n,\alpha}(\beta_{S_0}\in\cdot\mid S_\beta=S_0,Y)$ is also well defined on this event, because the oracle likelihood is bounded above and $\pi_p(s_0)>0$.

First localize the slab-weighted law on $D_n=\beta_{S_0}^0+\cB_{S_0}(r_n^0)$. Lemma~\ref{supp-le:oracle-mle} gives the MLE rate and Fisher equivalence of $\widehat J_{S_0}$, so \eqref{supp-eq:supp-likelihood-gaussian} and $r_n^0/\sqrt{s_0}\to\infty$ imply $Q_{n,\alpha}(D_n^c)=o_{\mP_{\beta^0}}(1)$. Integrating the lower likelihood sandwich over $\beta_{S_0}^0+\cB_{S_0}(\sqrt{s_0})$ and using $b_{g,n}=O(1)$ gives
\[
\int_{\beta_{S_0}^0+\cB_{S_0}(\sqrt{s_0})}
\e^{\alpha\ell_{n,S_0}(\beta_{S_0})}g_{S_0}(\beta_{S_0})\dif\beta_{S_0}
\ge
\e^{\alpha\ell_n(\beta^0)-O_{\mP_{\beta^0}}(s_0)}
g_{S_0}(\beta_{S_0}^0)|\mF_{S_0}^0|^{-1/2}.
\]
The truth-density lower bound, sparse Fisher geometry and $\log n=O(\log p)$ bound the reciprocal of the last two factors by $\exp(Cs_0\log p)$. The radial bound \eqref{supp-eq:rec-outside} therefore yields
\[
P_{n,\alpha}(D_n^c)
\le\exp\{O_{\mP_{\beta^0}}(s_0)+Cs_0\log p-\alpha c_{\rm rad}(r_n^0)^2\}
=o_{\mP_{\beta^0}}(1).
\]

For a fixed $M>0$, put
\[
E_n(M)=\{\beta_{S_0}:\|\beta_{S_0}-\beta_{S_0}^0\|_2\le M\sqrt{s_0/n}\}.
\]
These balls are contained in $D_n$ for all large $n$. The MLE score representation and Fisher equivalence, followed by the second-moment bound for a Gaussian vector, give
\[
\lim_{M\to\infty}\limsup_{n\to\infty}
\mE_{\beta^0}G_{n,\alpha}\{E_n(M)^c\}=0.
\]
Indeed, on the information and MLE event, the Gaussian tail probability is bounded by $C\{1+\|Z_{S_0}\|_2^2/s_0\}/M^2$, and the complementary event has probability tending to zero. Equation~\eqref{supp-eq:supp-likelihood-gaussian} transfers this tightness to $Q_{n,\alpha}$.

Conditionally on $D_n$, the density ratio of $P_{n,\alpha}$ to $Q_{n,\alpha}$ lies between $\e^{-2b_{g,n}}$ and $\e^{2b_{g,n}}$. Since $b_{g,n}=O(1)$ and both laws place probability tending to one on $D_n$, the same tightness holds for $P_{n,\alpha}$. Write $b_{g,n}^0(M)$ for the supremum of the absolute log-density ratio in Assumption~\ref{ass:oracle-laplace}. On $E_n(M)$, that oracle flatness clause of Assumption~\ref{ass:oracle-laplace} gives
\[
\|P_{n,\alpha}(\cdot\mid E_n(M))
-Q_{n,\alpha}(\cdot\mid E_n(M))\|_{\TV}
\le\e^{2b_{g,n}^0(M)}-1=o(1).
\]
Consequently,
\[
\|P_{n,\alpha}-Q_{n,\alpha}\|_{\TV}
\le P_{n,\alpha}\{E_n(M)^c\}
+Q_{n,\alpha}\{E_n(M)^c\}
+\e^{2b_{g,n}^0(M)}-1.
\]
Take $n\to\infty$ first and then $M\to\infty$. The preceding tightness and flatness bounds show that this total-variation distance tends to zero in $\mP_{\beta^0}$-probability. Combining it with \eqref{supp-eq:supp-likelihood-gaussian} proves the proposition.
\end{proof}

\subsection{Full oracle theorem and contraction}

\begin{proof}[Proof of Theorem~\ref{thm:oracle-bvm}]
The full-posterior total-variation error is at most $1-\Pi_{n,\alpha}(S_0\mid Y)$ plus the conditional approximation error. These vanish by Theorem~\ref{thm:selection}, proved in Subsection~\ref{supp-sec:overfitting}, and Proposition~\ref{prop:conditional-bvm}.
\end{proof}

\begin{proof}[Proof of Remark~\ref{rem:fisher-information-form}]
By \eqref{supp-eq:rec-mle-score-linearization}, $\|B_n-\mI_{s_0}\|_{\rm op}\le\varepsilon_n$ with probability tending to one. The Gaussian Kullback--Leibler divergence for replacing $\widehat J_{S_0}$ by $\mF_{S_0}^0$ is therefore at most $C s_0\varepsilon_n^2=o(1)$. Pinsker's inequality and Theorem~\ref{thm:oracle-bvm} prove the claim.
\end{proof}

\begin{proof}[Proof of Corollary~\ref{cor:posterior-contraction}]
Let $G_{n,\alpha}$ denote the oracle Gaussian law in Theorem~\ref{thm:oracle-bvm}. By \eqref{supp-eq:rec-mle-score-linearization} and Markov's inequality, on the event in Lemma~\ref{supp-le:oracle-mle},
\[
G_{n,\alpha}\left\{
\|(\mF_{S_0}^0)^{1/2}
(\beta_{S_0}-\beta_{S_0}^0)\|_2>
L_n\sqrt{s_0}\right\}
\le\frac{C\{s_0+\|Z_{S_0}\|_2^2\}}{L_n^2s_0}
=o_{\mP_{\beta^0}}(1).
\]
The exceptional event has probability $o(1)$. Since $G_{n,\alpha}$ is supported on $\{\beta_{S_0^c}=0\}$ and Assumption~\ref{ass:sparse-fisher} gives $\|(\mF_{S_0}^0)^{1/2}v\|_2\ge\sqrt{nc_F}\|v\|_2$, applying the bound with $\sqrt{c_F}L_n$ gives the full-vector Euclidean tail at radius $L_n\sqrt{s_0/n}$. The total-variation approximation transfers it to the posterior.
\end{proof}

\subsection{Credible sets and frequentist coverage}
\label{supp-sec:coverage}

\begin{proof}[Proof of Corollary~\ref{cor:coverage}]
Under the oracle Gaussian law, the quadratic form in \eqref{eq:oracle-credible-ellipsoid} is exactly $\chi_{s_0}^2$ and the standardized coordinate in \eqref{eq:oracle-marginal-interval} is standard normal. The two posterior credibility statements follow from Theorem~\ref{thm:oracle-bvm}.

When $s_0=d$ is fixed, the vector Lyapunov condition and the Cram\'er--Wold device give $Z_{S_0}=\sum_i\zeta_i\rightsquigarrow\mathcal N_d(0,\mI_d)$. Equation \eqref{supp-eq:rec-mle-score-linearization} then yields
\[
(\widehat\beta_{S_0}-\beta_{S_0}^0)^{\top}
\widehat J_{S_0}
(\widehat\beta_{S_0}-\beta_{S_0}^0)
=Z_{S_0}^{\top}A_n^{-1}B_nA_n^{-1}Z_{S_0}
\rightsquigarrow\chi_d^2.
\]
The same argument gives a standard-normal limit for each studentized coordinate,
\[
\frac{\widehat\beta_j-\beta_j^0}
{\{e_j^\top\widehat J_{S_0}^{-1}e_j\}^{1/2}}
=\frac{v_j^{\top}A_n^{-1}Z_{S_0}}
{\{v_j^{\top}B_n^{-1}v_j\}^{1/2}}.
\]
If $s_0$ grows, the directional Lyapunov condition gives $v_j^{\top}Z_{S_0}\rightsquigarrow\mathcal N(0,1)$. Since $\varepsilon_n\sqrt{s_0}=o(1)$, \eqref{supp-eq:rec-mle-score-linearization} gives
\[
|v_j^{\top}(A_n^{-1}-\mI_{s_0})Z_{S_0}|
=O_{\mP_{\beta^0}}(\varepsilon_n\sqrt{s_0})=o_{\mP_{\beta^0}}(1),
\qquad
v_j^{\top}B_n^{-1}v_j=1+o_{\mP_{\beta^0}}(1).
\]
Thus the directional studentized limit also follows. Finally, the coverage events reduce respectively to a Wald statistic bounded by $q_{d,1-\gamma}/\alpha$ and an absolute studentized statistic bounded by $z_{1-\gamma/2}/\sqrt\alpha$, which proves the displayed formulas.
\end{proof}

\begin{proof}[Proof of Corollary~\ref{cor:selected-coverage}]
Whenever $\Pi_{n,\alpha}(S_\beta=S_0\mid Y)>1/2$, the unique posterior modal support is $S_0$. Theorem~\ref{thm:selection} therefore gives $\mP_{\beta^0}(\widehat S_\alpha=S_0)\to1$. On the intersection of this event with the event in Lemma~\ref{supp-le:oracle-mle}, the selected-support regions equal their oracle counterparts. The discrepancy between each pair of posterior probabilities or coverage indicators is thus bounded by the indicator of the complementary event. Corollary~\ref{cor:coverage} gives all stated limits. Conditioning on selection does not change a limit because the conditioning event has $\mP_{\beta^0}$-probability tending to one.
\end{proof}

\section{Verification of assumptions}
\label{supp-sec:model-verification}

\subsection{Product-slab verification}
\label{supp-sec:prior-verification}

\begin{proof}[Proof of Proposition~\ref{prop:slab-verification}]
For the Gaussian slab, the truth-density lower bound follows from $\|\beta^0\|_\infty^2=O(\log p)$, and for $T\in\cS_n^+$,
\[
|\log g_T(\beta_T^0+u)-\log g_T(\beta_T^0)|
\le C\{\|\beta^0\|_2\|u\|_2+\|u\|_2^2\}.
\]
The assumed radius bound therefore gives bounded selection variation. At the oracle radius $t_n=M\sqrt{s_0/n}$, the right-hand side is $o(1)$ because $t_n/\rho_n=M\sqrt{c_Fs_0}/r_n^0\to0$ and $\rho_n\|\beta^0\|_2+\rho_n^2=O(1)$.

For the Laplace slab, the truth-density lower bound follows from $\|\beta^0\|_\infty=O(\log p)$, while the reverse triangle inequality gives
\[
|\log g_T(\beta_T^0+u)-\log g_T(\beta_T^0)|
\le\lambda\|u\|_1\le\lambda\sqrt{|T|}\|u\|_2.
\]
For the Student-$t$ slab, the assumed truth envelope gives the lower bound, and
\[
\sup_x|\{\log g(x)\}'|\le\frac{\nu+1}{2\sigma\sqrt\nu}
\]
gives the same bound with another fixed constant. Thus $\sqrt{s_n^\star}\rho_n=O(1)$ supplies bounded variation on the selection neighbourhoods. On $S_0$ at radius $t_n$, the variation is at most $C\sqrt{s_0}t_n$, which tends to zero because
\[
\sqrt{s_0}t_n
\le\sqrt{s_n^\star}\rho_n\,(t_n/\rho_n)=o(1).\qedhere
\]
\end{proof}

\begin{samepage}
\subsection{General fixed-design likelihood tools}
\label{supp-sec:fixed-design-verification}

The criterion below uses average spread to obtain the fixed-design rates used in the main comparison. Assumption~\ref{ass:sparse-fisher} is imposed on the fixed design, and the likelihood verification establishes Assumptions~\ref{ass:score} and \ref{ass:truth-lan} and the likelihood part of Assumption~\ref{ass:oracle-laplace}. Table~\ref{tab:six-models} in the main text records the model-specific restrictions.
\par
\end{samepage}

\begin{samepage}
\medskip
\noindent\textit{Fixed designs.}\quad
We express the likelihood conditions through quantities determined by the design and model. Let $w_i^0$ denote the $i$th diagonal entry of $\mW_0$ and put $z_{i,S}=(\mF_S^0)^{-1/2}X_{i,S}$. For $s_0\le s\le s_n^\star$, define the sparse influence and weighted leverage by
\[
q_\star(s)=\sup_{S\in\cS_n^+,\,|S|\le s}\max_i\|z_{i,S}\|_2,
\qquad
\ell_\star(s)=\sup_{S\in\cS_n^+,\,|S|\le s}\max_i
w_i^0\|z_{i,S}\|_2^2.
\]
\par
\end{samepage}
The first quantity controls how far a Fisher ball moves an individual linear predictor. The second controls the largest single-observation contribution to a Fisher-normalized quadratic form. To measure aggregate spread rather than only the worst row, define
\[
\mathfrak t_\star(s)=
\sup_{\substack{S\in\cS_n^+\\|S|\le s\\\|a\|_2=\|u\|_2=1}}
\sum_{i=1}^nw_i^0|z_{i,S}^\top a|(z_{i,S}^\top u)^2,
\qquad
\mathfrak v_\star(s)=
\sup_{\substack{S\in\cS_n^+\\|S|\le s\\\|u\|_2=1}}
\left\{\sum_{i=1}^n(w_i^0)^2(z_{i,S}^\top u)^4\right\}^{1/2}.
\]
These are deterministic design quantities once the truth-Fisher weights are specified. The average-spread condition
\begin{equation}\label{supp-eq:primitive-average-spread}
\mathfrak t_\star(s_n^\star)+\mathfrak v_\star(s_n^\star)
\le Cn^{-1/2}
\end{equation}
says that Fisher-standardized third- and fourth-order contributions are distributed across observations.

We also use the baseline influence and leverage bounds
\begin{equation}\label{supp-eq:primitive-design-rates}
q_\star(s_n^\star)\le C\sqrt{s_n^\star/n},\qquad
\ell_\star(s_n^\star)\le Cs_n^\star/n.
\end{equation}
For example, these bounds follow from sparse Fisher geometry, uniformly bounded truth weights and uniformly bounded design entries. This special case and direct moment criteria for \eqref{supp-eq:primitive-average-spread} are detailed below. Since $s_n^\star=(K_0+1)s_0$, all six models require
\begin{equation}\label{supp-eq:primitive-rate-base}
\frac{s_0^2\log p}{n}\longrightarrow0.
\end{equation}
Table~\ref{supp-tab:fixed-design-rates} gives the remaining model-specific growth restrictions. Gaussian likelihoods are quadratic, so average spread and an additional rate are unnecessary in that case.
\begin{table}[!ht]
\centering
\setlength{\tabcolsep}{3pt}
\tbl{Additional fixed-design rates under average spread.}{%
\footnotesize
\begin{tabular}{@{}p{0.62\linewidth}p{0.28\linewidth}@{}}
\toprule
Models & additional rate\\
\midrule
Gaussian & none\\
Logistic, Poisson and bounded-window probit & $s_0^3/(n\log p)\to0$\\
Gamma and negative binomial & $s_0^3/n\to0$\\
\bottomrule
\end{tabular}}
\label{supp-tab:fixed-design-rates}
\end{table}

\begin{samepage}
\begin{proposition}[Verification of likelihood conditions]
\label{supp-prop:primitive-bridge}
Consider any model in Table~\ref{tab:six-models}. Suppose that Assumption~\ref{ass:sparse-fisher} and conditions \eqref{supp-eq:primitive-design-rates} and \eqref{supp-eq:primitive-rate-base} hold, together with the corresponding model-specific restriction. For a non-Gaussian model, also suppose \eqref{supp-eq:primitive-average-spread} and the applicable rate in Table~\ref{supp-tab:fixed-design-rates}. Then Assumptions~\ref{ass:score} and \ref{ass:truth-lan} and the likelihood part of Assumption~\ref{ass:oracle-laplace} hold.
\end{proposition}
\end{samepage}

\begin{proof}[Proof of Proposition~\ref{prop:simple-fixed-design}]
For every $S\in\cS_n^+$, Assumption~\ref{ass:sparse-fisher} and the entry bounds in \eqref{eq:simple-design-entries} give
\[
 \|z_{i,S}\|_2
 \le\frac{\|X_{i,S}\|_2}{\sqrt{nc_F}}
 \le\frac{K_X}{\sqrt{c_F}}\sqrt{\frac{|S|}{n}},
 \qquad
 w_i^0\|z_{i,S}\|_2^2
 \le\frac{w_i^0\|X_{i,S}\|_2^2}{nc_F}
 \le\frac{K_V^2}{c_F}\frac{|S|}{n}.
\]
Taking suprema proves \eqref{supp-eq:primitive-design-rates}. For a non-Gaussian model, the moment bounds in \eqref{eq:simple-spread-moments} imply \eqref{supp-eq:primitive-average-spread}. H\"older's inequality with exponents $3$ and $3/2$ bounds $\mathfrak t_\star(s_n^\star)$ by the first cubic moment, while the second moment bound is exactly $\mathfrak v_\star(s_n^\star)^2\le C/n$. The growth conditions in \eqref{eq:simple-fixed-growth} are exactly \eqref{supp-eq:primitive-rate-base} and the applicable row of Table~\ref{supp-tab:fixed-design-rates}. An upper Poisson predictor bound is one sufficient way to verify the weighted-entry condition, but neither the simple nor the general criterion requires that window. Proposition~\ref{supp-prop:primitive-bridge} therefore proves the claim. Its proof chooses $r_n^0=M_n\sqrt{s_n^\star\log p}$ with $M_n\to\infty$ sufficiently slowly. Since $s_n^\star\asymp s_0$, the baseline growth condition also permits $\sqrt{s_n^\star}r_n^0/\sqrt n\to0$, which is sufficient for the separate slab verification at the same radius.
\end{proof}

\paragraph{Common likelihood bounds}
Write $\mu_i^0=b'\{\xi(\eta_i^0)\}$. The observed information decomposes as
$-\ell_i''(\eta)=A_i(\eta)-B_i(\eta)(Y_i-\mu_i^0)$, where
\[
 A_i(\eta)=\tau_i^{-1}\left[
 b''\{\xi(\eta)\}\{\xi'(\eta)\}^2
 -\xi''(\eta)\{\mu_i^0-b'(\xi(\eta))\}\right],
 \qquad B_i(\eta)=\tau_i^{-1}\xi''(\eta),
\]
and $A_i(\eta_i^0)=w_i^0$. The response condition is
\begin{equation}
 \log\mE_{\beta^0}e^{t(Y_i-\mu_i^0)}
 \le\frac{\sigma_i^2t^2}{2(1-c_{Y,i}|t|)},\qquad
 \sigma_i^2=\tau_i b''\{\xi(\eta_i^0)\},\qquad |t|c_{Y,i}<1.
 \label{supp-eq:supp-response-bernstein}
\end{equation}
We use the following uniform model conditions, denoted by \textnormal{(E-GLM)}. The likelihood is concave on every support, \eqref{supp-eq:supp-response-bernstein} holds, and, for a fixed $C$,
\begin{equation}
 c_{Y,i}\tau_i^{-1}|\xi'(\eta_i^0)|\le C,\qquad
 \sup_{|\eta-\eta_i^0|\le1}
 \frac{|A_i'(\eta)|+(\sigma_i+c_{Y,i})
       \sum_{j=0}^2|B_i^{(j)}(\eta)|}{w_i^0}\le C
 \label{supp-eq:supp-model-envelopes}
\end{equation}
uniformly in $i$. Here $A_i$ is assumed continuously differentiable and $B_i$ twice continuously differentiable on these tubes. Subsection~\ref{supp-sec:six-model-cards} verifies these bounds for all six models. For random designs, the constants are uniform on the good-design events.

Choose subpolynomial, slowly diverging $M_n,M_n^0$ and put
$r_n^0=M_n\sqrt{s_n^\star\log p}$ and $r_n=M_n^0\sqrt{s_0}$.
For this sufficient verification criterion, we choose the radii to satisfy
\begin{equation}
 \frac{\sqrt{s_n^\star}r_n^0}{\sqrt n}\to0,\qquad
 q_\star(s_n^\star)r_n^0\to0,\qquad
 q_\star(s_0)r_n\to0,\qquad r_n=o(r_n^0).
 \label{supp-eq:supp-model-radii}
\end{equation}
One common slow sequence $M_n=M_n^0$ suffices whenever the corresponding rates vanish. The standing condition $\log n=O(\log p)$ and these radius bounds give the net-cardinality estimates used below.

For a sufficiently large fixed $C$, define the selection error and oracle modulus bounds
\begin{align}
 \eta_{\rm sel,n}
 &=C\left\{\mathfrak t_\star(s_n^\star)r_n^0
 +\mathfrak v_\star(s_n^\star)\sqrt{s_n^\star\log p}
 +\ell_\star(s_n^\star)s_n^\star\log p+q_\star(s_n^\star)\right\},
 \label{supp-eq:supp-selection-grid-rate}\\
 \mathfrak k_n
 &=C\Bigl\{\mathfrak t_\star(s_0)
 +q_\star(s_0)\mathfrak v_\star(s_0)\sqrt{s_0\log(en)}
 \nonumber\\*
 &\qquad+q_\star(s_0)\ell_\star(s_0)s_0\log(en)
 +\frac{q_\star(s_0)\mathfrak t_\star(s_0)}{r_n^2}\Bigr\}.
 \label{supp-eq:supp-oracle-grid-modulus}
\end{align}
For canonical links, $B_i\equiv0$, so retain only the first term in each bound. In particular, the Poisson likelihood calculation uses only the influence and cubic-spread bounds, although the common fixed-design criterion also imposes leverage and fourth-moment bounds. For Gaussian regression both bounds are zero. The probit refinement is given below.

\begin{proposition}[Likelihood verification]
\label{supp-pr:truth-centred-verification}
Suppose Assumption~\ref{ass:sparse-fisher}, \textnormal{(E-GLM)} and \eqref{supp-eq:supp-model-radii} hold. Then Assumptions~\ref{ass:score} and \ref{ass:truth-lan} and the likelihood part of Assumption~\ref{ass:oracle-laplace} hold, with $\bar\kappa_n^0=C\mathfrak k_n$ in Assumption~\ref{ass:oracle-laplace} and failure probabilities uniform whenever the model and design bounds are uniform, provided
\begin{equation}
 \eta_{\rm sel,n}(s_0+\log p)=o(\log p),\qquad
 s_0\mathfrak k_n=o(1).
 \label{supp-eq:supp-likelihood-requirements}
\end{equation}
\end{proposition}

\begin{proof}
Independence, \eqref{supp-eq:supp-response-bernstein} and Fisher normalization give
\begin{equation}
 \log\mE_{\beta^0}\exp(\lambda u^\top Z_S)
 \le\frac{\lambda^2}{2(1-|\lambda|d_n)},\qquad
 d_n=\sup_{S,u,i}c_{Y,i}\tau_i^{-1}
       |\xi'(\eta_i^0)z_{i,S}^\top u|
 \le Cq_\star(s_n^\star),
 \label{supp-eq:supp-score-domain}
\end{equation}
for unit $u$ and $|\lambda|d_n<1$. The radius conditions imply
$d_n\sqrt{s_n^\star\log p}=o(1)$, proving Assumption~\ref{ass:score}. For Gaussian scores take $d_n=0$.

For Fisher coordinates $h$, let
\[
 \mathcal I_S(h)=
 \sum_i\{A_i(\eta_i^0+z_{i,S}^\top h)
 -B_i(\eta_i^0+z_{i,S}^\top h)(Y_i-\mu_i^0)\}
 z_{i,S}z_{i,S}^\top.
\]
The mean-value theorem bounds the deterministic part of
$\mathcal I_S(h)-\mI$ by $C\mathfrak t_\star(s_n^\star)r_n^0$ uniformly on $\|h\|_2\le r_n^0$.
For the residual part, take a Euclidean $\{\log(np)\}^{-1}$-net of each ball and a $1/4$-net of its unit sphere. The logarithm of their total size over $\cS_n^+$ is $O(s_n^\star\log p)$. At a fixed pair $(h,u)$, the model envelopes give the Bernstein variance and scale bounds
\[
\begin{aligned}
 \sum_i\sigma_i^2 B_i(\eta_i^0+z_{i,S}^\top h)^2
 (z_{i,S}^\top u)^4
 &\le C\mathfrak v_\star(s_n^\star)^2,\\
 \max_i c_{Y,i}|B_i(\eta_i^0+z_{i,S}^\top h)|
 (z_{i,S}^\top u)^2
 &\le C\ell_\star(s_n^\star).
\end{aligned}
\]
Bernstein's inequality and a union bound give the second and third terms of \eqref{supp-eq:supp-selection-grid-rate}. The symmetric-matrix net inequality extends the unit directions. A response union bound also gives, with probability tending to one,
\begin{equation}
 |Y_i-\mu_i^0|
 \le C\{\sigma_i\sqrt{\log(np)}+c_{Y,i}\log(np)\},
 \qquad i=1,\ldots,n.
 \label{supp-eq:supp-response-event}
\end{equation}
The bound on $B_i'$ in \eqref{supp-eq:supp-model-envelopes}, the mesh size and
$\sum_iw_i^0(z_{i,S}^\top u)^2=1$ bound the interpolation error by
$Cq_\star(s_n^\star)$. Hence
\begin{equation*}
 \sup_{S\in\cS_n^+}\sup_{\|h\|_2\le r_n^0}
 \|\mathcal I_S(h)-\mI\|_{\rm op}\le\eta_{\rm sel,n}
\end{equation*}
with probability tending to one. The first condition in \eqref{supp-eq:supp-likelihood-requirements} and concavity now give Assumption~\ref{ass:truth-lan}, including the radial bound of Lemma~\ref{supp-le:rec-lan}.

For the oracle modulus, the directional derivative has quadratic form
\[
 u^\top\partial_a\mathcal I_{S_0}(h)u
 =\sum_i\{A_i'(\eta_i^0+z_{i,S_0}^\top h)
 -B_i'(\eta_i^0+z_{i,S_0}^\top h)(Y_i-\mu_i^0)\}
 (z_{i,S_0}^\top a)(z_{i,S_0}^\top u)^2.
\]
Its deterministic part is bounded by $C\mathfrak t_\star(s_0)$. For the residual part, take a Euclidean net of the $2r_n$ ball with mesh
$\{r_n^2\log(2n)\}^{-1}$ and $1/4$-nets for $a$ and $u$. Their logarithmic size is $O\{s_0\log(en)\}$. Since $|z_{i,S_0}^\top a|\le q_\star(s_0)$, the Bernstein variance and scale satisfy
\[
\begin{aligned}
 \sum_i\sigma_i^2 B_i'(\eta_i^0+z_{i,S_0}^\top h)^2
 (z_{i,S_0}^\top a)^2(z_{i,S_0}^\top u)^4
 &\le Cq_\star(s_0)^2\mathfrak v_\star(s_0)^2,\\
 \max_i c_{Y,i}|B_i'(\eta_i^0+z_{i,S_0}^\top h)
 z_{i,S_0}^\top a|(z_{i,S_0}^\top u)^2
 &\le Cq_\star(s_0)\ell_\star(s_0).
\end{aligned}
\]
Bernstein's inequality, a union bound and the net extensions in the linear direction $a$ and the quadratic direction $u$ give the second and third terms of \eqref{supp-eq:supp-oracle-grid-modulus}. Using \eqref{supp-eq:supp-response-event} with $\log(en)$ in place of $\log(np)$ and the bound on $B_i''$, the interpolation error is at most
$Cq_\star(s_0)\mathfrak t_\star(s_0)/r_n^2$. Therefore
\begin{equation}
 \sup_{\|h\|_2\le2r_n}\sup_{\|a\|_2=1}
 \|\partial_a\mathcal I_{S_0}(h)\|_{\rm op}\le\mathfrak k_n
 \label{supp-eq:supp-oracle-lipschitz}
\end{equation}
with probability tending to one. Integrating along a segment gives the corresponding Lipschitz bound.

Finally, Lemma~\ref{supp-le:oracle-mle} gives the unique oracle MLE at truth-Fisher distance $O_{\mP_{\beta^0}}(\sqrt{s_0})$, with $\widehat J_{S_0}$ Fisher-equivalent. Since $r_n/\sqrt{s_0}\to\infty$, its $\widehat J_{S_0}$-ball of radius $r_n$ lies in the truth-Fisher ball of radius $2r_n$ with probability tending to one. Changing between these equivalent normalizations in \eqref{supp-eq:supp-oracle-lipschitz} yields $\kappa_n^0\le C\mathfrak k_n$. Intersecting the score, information and oracle curvature events gives probability tending to one. The second condition in \eqref{supp-eq:supp-likelihood-requirements} gives $s_0\bar\kappa_n^0\to0$, completing the likelihood part of Assumption~\ref{ass:oracle-laplace}. Each bound above is uniform under the stated uniform model and design conditions.
\end{proof}

\subsection{Model-specific likelihood verification}
\label{supp-sec:six-model-cards}

We verify \textnormal{(E-GLM)} using the links and Fisher weights in Table~\ref{tab:six-models}. The shrinking predictor tubes in \eqref{supp-eq:supp-model-radii} are eventually contained in $|\eta-\eta_i^0|\le1$.

\begin{enumerate}[leftmargin=*,label=\textnormal{(\roman*)}]
\item \emph{Gaussian regression.}
Here $b(\theta)=\theta^2/2$, $A_i=\tau_i^{-1}$ and $B_i=0$. The normalized score is Gaussian and the likelihood is quadratic and concave. Thus $d_n=\eta_{\rm sel,n}=\mathfrak k_n=0$.

\item \emph{Logistic regression.}
The centred Bernoulli mgf satisfies \eqref{supp-eq:supp-response-bernstein} with $c_{Y,i}=1/3$ and $\sigma_i^2=w_i^0$. Since $B_i=0$ and $|(\log w)'|\le1$, the normalized derivative bound holds on the unit tube. The canonical likelihood is concave. Consequently,
\begin{equation}
 \eta_{\rm sel,n}=C\mathfrak t_\star(s_n^\star)r_n^0,\qquad
 \mathfrak k_n=C\mathfrak t_\star(s_0).
 \label{supp-eq:supp-canonical-rates}
\end{equation}

\item \emph{Poisson regression.}
The centred mgf $\exp\{\mu_i^0(e^t-1-t)\}$ satisfies \eqref{supp-eq:supp-response-bernstein} with $\sigma_i^2=\mu_i^0$ and $c_{Y,i}=1/3$. Here $A_i(\eta)=e^\eta$, $B_i=0$ and $|A_i'(\eta)|/w_i^0\le e$ on the unit tube. The likelihood is concave and \eqref{supp-eq:supp-canonical-rates} applies without an upper truth-predictor bound. In particular, the score requires only a finite-domain Bernstein bound, not a sub-Gaussian response. Proposition~\ref{supp-prop:poisson-subgaussian-bridge} verifies the random design when the Fisher weights are unbounded.

\item \emph{Probit regression.}
The window $|\eta_i^0|\le L$ confines the unit tube to $[-L-1,L+1]$. All normalized derivatives in \eqref{supp-eq:supp-model-envelopes} are therefore bounded, and the Bernoulli response satisfies \eqref{supp-eq:supp-response-bernstein}. Concavity follows from that of $\log\Phi$ and $\log(1-\Phi)$. Since $|Y_i-\mu_i^0|\le1$, Hoeffding's inequality removes the linear Bernstein term in the selection bound. The observed third derivative is also deterministically bounded by $C_Lw_i^0$, so \eqref{supp-eq:supp-oracle-lipschitz} holds with $C_L\mathfrak t_\star(s_0)$. We may therefore use
\begin{equation*}
 \eta_{\rm sel,n}
 =C_L\{\mathfrak t_\star(s_n^\star)r_n^0
 +\mathfrak v_\star(s_n^\star)\sqrt{s_n^\star\log p}
 +q_\star(s_n^\star)\},\qquad
 \mathfrak k_n=C_L\mathfrak t_\star(s_0).
\end{equation*}

\item \emph{Gamma regression.}
Here $b(\theta)=-\log(-\theta)$, $\xi(\eta)=-e^{-\eta}<0$,
$A_i(\eta)=\tau_i^{-1}e^{\eta_i^0-\eta}$ and
$B_i(\eta)=-\tau_i^{-1}e^{-\eta}$. The centred Gamma mgf gives
$\sigma_i^2=\tau_i(\mu_i^0)^2$ and $c_{Y,i}=\tau_i\mu_i^0$.
Thus $c_{Y,i}\tau_i^{-1}\xi'(\eta_i^0)=1$, and the remaining normalized derivatives in \eqref{supp-eq:supp-model-envelopes} are bounded on the unit tube. No truth-predictor bound is needed. The identity
$\ell_i(\eta)=\tau_i^{-1}(-Y_ie^{-\eta}-\eta)+C_i$, where $C_i$ does not depend on $\eta$, proves concavity. The full bounds \eqref{supp-eq:supp-selection-grid-rate}--\eqref{supp-eq:supp-oracle-grid-modulus} apply.

\item \emph{Negative-binomial regression.}
Here $b(\theta)=-\kappa\log(1-e^\theta)$ and
$\xi(\eta)=\eta-\log(\kappa+e^\eta)<0$. The observed information gives
\[
 B_i(\eta)=-\frac{\kappa e^\eta}{(\kappa+e^\eta)^2},\qquad
 A_i(\eta)=-(\kappa+\mu_i^0)B_i(\eta).
\]
The centred negative-binomial mgf satisfies \eqref{supp-eq:supp-response-bernstein} with
$\sigma_i^2=\mu_i^0(\kappa+\mu_i^0)/\kappa$ and
$c_{Y,i}=(\kappa+\mu_i^0)/\kappa$. In particular,
$c_{Y,i}\xi'(\eta_i^0)=1$; the constant term in $c_{Y,i}$ also covers small means.
At the truth,
\[
 \frac{\sigma_i|B_i(\eta_i^0)|}{w_i^0}
 =\sqrt{\frac{\mu_i^0}{\kappa(\kappa+\mu_i^0)}}\le\kappa^{-1/2},
 \qquad
 \frac{c_{Y,i}|B_i(\eta_i^0)|}{w_i^0}=\kappa^{-1}.
\]
The first two derivatives of $B_i$ are bounded by $C_\kappa|B_i|$, and
$|(\log|B_i|)'|\le1$. These bounds extend to the unit tube and verify
\eqref{supp-eq:supp-model-envelopes}. Since $-\ell_i''(\eta)=-(Y_i+\kappa)B_i(\eta)\ge0$, the likelihood is concave.
The full bounds \eqref{supp-eq:supp-selection-grid-rate}--\eqref{supp-eq:supp-oracle-grid-modulus} therefore apply without a truth-predictor window.
\end{enumerate}

\begin{proof}[Proof of Proposition~\ref{supp-prop:primitive-bridge}]
Choose a common $M_n=M_n^0\to\infty$ sufficiently slowly and set
$r_n^0=M_n\sqrt{s_n^\star\log p}$ and $r_n=M_n\sqrt{s_0}$.
Since $s_n^\star\asymp s_0$, the rates in Table~\ref{supp-tab:fixed-design-rates} and
\eqref{supp-eq:primitive-rate-base} permit
\begin{equation}
 M_ns_n^\star\sqrt{\log p/n}\to0,\qquad
 \frac{M_n(s_n^\star)^{3/2}}{\sqrt{n\log p}}\to0
 \quad\text{for non-Gaussian models}.
 \label{supp-eq:supp-fixed-design-radii}
\end{equation}
Only the first condition is needed for Gaussian regression. The design bounds
\eqref{supp-eq:primitive-design-rates} then give \eqref{supp-eq:supp-model-radii}.
Under average spread, the selection bound for logistic, Poisson and probit is
$O\{M_n\sqrt{s_n^\star\log p/n}\}$.
For Gamma and negative-binomial regression it has the additional term
$O\{(s_n^\star)^2\log p/n\}$. Therefore
\[
 \frac{(s_0+\log p)\eta_{\rm sel,n}}{\log p}
 \lesssim M_n\sqrt{\frac{s_n^\star\log p}{n}}
 +\frac{M_n(s_n^\star)^{3/2}}{\sqrt{n\log p}}+o(1),
\]
where the last term includes
$(s_n^\star)^2\log p/n+(s_n^\star)^3/n$ for Gamma and negative-binomial regression.
All terms vanish under the applicable rates. In particular, $\eta_{\rm sel,n}(s_0+\log p)=o(\log p)$ and $\log\mathfrak E_{n,\alpha}=o(\log p)$.

For the oracle bound, logistic, Poisson and probit have
$s_0\mathfrak k_n=O(s_0/\sqrt n)=o(1)$, whereas Gamma and negative-binomial regression give
\[
 s_0\mathfrak k_n
 =O\left\{\frac{s_0}{\sqrt n}
 +\frac{s_0^2\sqrt{\log(en)}}{n}
 +\frac{s_0^{7/2}\log(en)}{n^{3/2}}\right\}=o(1).
\]
The last limits follow from $s_0^3/n\to0$ and
$n^{-1/3}\log(en)\to0$. Choose $M_n$ slowly enough that
$r_n\mathfrak k_n=o(1)$ as well. Gaussian regression has zero information errors.
Proposition~\ref{supp-pr:truth-centred-verification} now proves the claim.
For logistic, Poisson and probit the same bounds also give
\begin{equation}
 \sqrt{s_0}\eta_{\rm sel,n}
 \lesssim M_ns_n^\star\sqrt{\log p/n}=o(1),
 \label{supp-eq:supp-fixed-calibration}
\end{equation}
which verifies the true-Fisher replacement and growing-coordinate linearization condition.
\end{proof}

\subsection{Random-design verification}
\label{supp-sec:random-design-verification}

The random-design verification is stated in terms of the weighted rows that generate Fisher information, because sub-Gaussianity of $X_i$ alone need not control $\{w_i^0\}^{1/2}X_i$ in a nonlinear generalized linear model. Write $\mathbb E_{\mX}$ for expectation under $\mathbb P_{\mX}$, and $\mP_{\beta^0}(\,\cdot\mid\mX)$ and $\mE_{\beta^0}(\,\cdot\mid\mX)$ for conditional-response probability and expectation. For a scalar random variable $Z$ and $\nu>0$, write $\|Z\|_{\psi_\nu}=\inf\{c>0:\mathbb E_{\mX}\exp(|Z/c|^\nu)\le2\}$ and put $\mathcal H_n(s)=s\log(ep/s)+\log n$.

\begin{proposition}[Sub-Gaussian random-design geometry]
\label{supp-prop:random-design-geometry}
Suppose $X_1,\ldots,X_n$ are independent and put $V_i=(w_i^0)^{1/2}X_i$. For fixed $K_X,K_V<\infty$, assume
\[
\sup_i\sup_{\|u\|_2=1}\|u^\top X_i\|_{\psi_2}\le K_X,
\qquad
\sup_i\sup_{\|u\|_2=1}\|u^\top V_i\|_{\psi_2}\le K_V.
\]
Let $\overline\Sigma_n=n^{-1}\sum_i\mathbb E_{\mX}(V_iV_i^\top)$ and suppose its eigenvalues on every coordinate subspace of size at most $s_n^\star$ lie between fixed positive constants. If $s_n^\star=o(p)$ and $\mathcal H_n(s_n^\star)^2=o(n)$, then, with $\mathbb P_{\mX}$-probability tending to one, Assumption~\ref{ass:sparse-fisher} and the two empirical moment bounds in \eqref{eq:simple-spread-moments} hold. On the same event, for $s\in\{s_0,s_n^\star\}$, the following bounds hold with implied constants depending only on the displayed sub-Gaussian and population eigenvalue constants,
\begin{equation}\label{supp-eq:random-design-rates}
q_\star(s)\lesssim\sqrt{\frac{\mathcal H_n(s)}{n}},\qquad
\ell_\star(s)\lesssim\frac{\mathcal H_n(s)}{n},\qquad
\mathfrak t_\star(s)+\mathfrak v_\star(s)
\lesssim n^{-1/2}.
\end{equation}
\end{proposition}

\begin{samepage}
\begin{proposition}[Random-design likelihood bridge]
\label{supp-prop:random-design-bridge}
Consider the conditional generalized linear model given the random design. Under the corresponding restriction in Table~\ref{tab:six-models}, assume that the uniform generalized-linear-model envelope package \textnormal{(E-GLM)} in Subsection~\ref{supp-sec:fixed-design-verification} holds. Under the conditions of Proposition~\ref{supp-prop:random-design-geometry}, put $H_\star=\mathcal H_n(s_n^\star)$. Suppose $\log n=O(\log p)$. Then there are design events $\mathcal X_n$, with $\mathbb P_{\mX}(\mathcal X_n)\to1$, on which Assumption~\ref{ass:score} holds for the conditional response law given $\mX$, while the likelihood events required in Assumptions~\ref{ass:truth-lan} and \ref{ass:oracle-laplace} have conditional $\mP_{\beta^0}(\,\cdot\mid\mX)$-probability tending to one uniformly for $\mX\in\mathcal X_n$, provided
\begin{equation}\label{supp-eq:random-design-growth}
\frac{H_\star s_n^\star(s_n^\star+\log p)^2}
     {n\log p}\longrightarrow0.
\end{equation}
\end{proposition}
\end{samepage}

Propositions~\ref{supp-prop:random-design-geometry} and \ref{supp-prop:random-design-bridge} allow the fixed-design conclusions to be applied conditionally on $\mX$ and then integrated over the design. In the common regime $s_0\le\log p$ and $\log n=O(\log p)$, a sufficient form of \eqref{supp-eq:random-design-growth} is $s_0^2(\log p)^2=o(n)$. For fixed $s_0$, this reduces to $(\log p)^2=o(n)$. These are likelihood conditions. Corollary~\ref{cor:primitive-six-models} adds the prior, slab and beta-min requirements.

\begin{proposition}[Poisson sub-Gaussian-design bridge]
\label{supp-prop:poisson-subgaussian-bridge}
Consider the Poisson regression. Suppose $X_1,\ldots,X_n$ are independent, uniformly sub-Gaussian rows, \mbox{$\|\beta^0\|_2\le B<\infty$}, and, for fixed $0<c_X<C_X<\infty$,
\[
 c_X\le \frac1n\sum_{i=1}^n\mathbb E_{\mX}(u^\top X_i)^2
 \le C_X
\]
for every unit $u$ with $\|u\|_0\le s_n^\star$. Put $H_\star=\mathcal H_n(s_n^\star)$. Suppose also that $s_n^\star=o(p)$ and $\log n=O(\log p)$. Then there are design events $\mathcal X_n$, with $\mathbb P_{\mX}(\mathcal X_n)\to1$, on which Assumption~\ref{ass:sparse-fisher} holds, Assumption~\ref{ass:score} holds for the conditional response law given $\mX$, and the likelihood events required in Assumptions~\ref{ass:truth-lan} and \ref{ass:oracle-laplace} have conditional $\mP_{\beta^0}(\,\cdot\mid\mX)$-probability tending to one uniformly for $\mX\in\mathcal X_n$, provided
\begin{equation}\label{supp-eq:poisson-subgaussian-growth}
 \frac{H_\star s_n^\star\log p}{n}\longrightarrow0,
 \qquad
 \frac{s_n^\star(s_n^\star+\log p)^2}
      {n\log p}\longrightarrow0.
\end{equation}
\end{proposition}

Unlike Proposition~\ref{supp-prop:random-design-geometry}, Proposition~\ref{supp-prop:poisson-subgaussian-bridge} does not assume that the Fisher-weighted rows are sub-Gaussian, a property that fails for Gaussian covariates when $\beta^0\ne0$. Its proof instead controls the random Fisher matrix directly from the unweighted rows. Because Poisson is canonical, $\ell_\star$ and $\mathfrak v_\star$ are not needed.
For Gaussian, logistic, bounded-window probit, Gamma log-link and fixed-size negative-binomial regression, a uniform upper bound on $w_i^0$ makes the weighted-row condition in Proposition~\ref{supp-prop:random-design-geometry} automatic from sub-Gaussianity of $X_i$. The same route covers Poisson under $\sup_{i\le n}X_i^\top\beta^0\le L$ $\mathbb P_{\mX}$-almost surely, while Proposition~\ref{supp-prop:poisson-subgaussian-bridge} covers unbounded Fisher weights under bounded $\|\beta^0\|_2$. Thus Propositions~\ref{supp-prop:random-design-geometry}--\ref{supp-prop:random-design-bridge}, or the Poisson-specific bridge, provide random-design verifications for all six models.

\begin{proof}[Proof of Proposition~\ref{prop:simple-random-design}]
Put $L_n=\log p$. Under the standing conditions $s_n^\star=(K_0+1)s_0$ and $\log n=O(L_n)$, the restriction $s_0\le L_n$ gives $H_\star=\mathcal H_n(s_n^\star)\lesssim s_0L_n$ and $s_n^\star+L_n\lesssim L_n$. Consequently, \eqref{eq:simple-random-growth} implies
\[
 \frac{H_\star^2}{n}
 +\frac{H_\star s_n^\star(s_n^\star+L_n)^2}{nL_n}
 +\frac{H_\star s_n^\star L_n}{n}
 \lesssim\frac{s_0^2L_n^2}{n}\longrightarrow0.
\]
In the bounded-weight route, the stated model restrictions give $w_i^0\le W$ for a fixed $W<\infty$. Thus, for every unit $u$, $\|u^\top V_i\|_{\psi_2} \le\sqrt W\,\|u^\top X_i\|_{\psi_2}$. The weighted-row condition follows, and Propositions~\ref{supp-prop:random-design-geometry} and \ref{supp-prop:random-design-bridge} apply with the imposed population Fisher bounds. Subsection~\ref{supp-sec:six-model-cards} supplies the model-specific likelihood envelopes. In the unbounded-weight Poisson route, the preceding growth bound also verifies the first condition in \eqref{supp-eq:poisson-subgaussian-growth}. Moreover,
\[
 \frac{s_n^\star(s_n^\star+L_n)^2}{nL_n}
 \lesssim\frac{s_0L_n}{n}\longrightarrow0,
\]
so Proposition~\ref{supp-prop:poisson-subgaussian-bridge} applies. The corresponding good-design events and uniform conditional-response bounds prove the statement. The model-size prior, slab and signal conditions remain the separate requirements of Corollary~\ref{cor:primitive-six-models}.
\end{proof}

\begin{proof}[Proof of Proposition~\ref{supp-prop:random-design-geometry}]
For each coordinate set, fix a $1/8$-net of its unit sphere. The total number of points over sets of size at most $s$ is at most $\exp\{Cs\log(ep/s)\}$. For every fixed unit $u$, the centred variables $(u^\top V_i)^2-\mathbb E_{\mX}(u^\top V_i)^2$ have bounded $\psi_1$ norms, so Bernstein's inequality gives, for fixed $t>0$,
\begin{equation*}
 \mathbb P_{\mX}\left(
 \left|\frac1n\sum_i(u^\top V_i)^2
 -u^\top\overline\Sigma_nu\right|>t\right)
 \le2\exp(-c_tn).
\end{equation*}
Take $t$ to be a fixed fraction of the population eigenvalue lower bound. A union bound and the symmetric-matrix net extension give Assumption~\ref{ass:sparse-fisher}, since $s_n^\star\log(ep/s_n^\star)=o(n)$. Applied to $X_i$, the same argument gives
\begin{equation}
 \sup_{\substack{|S|\le s_n^\star\\\|u\|_2=1}}
 \frac1n\sum_{i=1}^n(X_{i,S}^\top u)^2\le C
 \label{supp-eq:supp-random-second-moment}
\end{equation}
with probability tending to one.

Uniform sub-Gaussianity gives $\sup_{i,u}\mathbb E_{\mX}|u^\top V_i|^4\le C$, and the centred fourth powers have uniformly bounded $\psi_{1/2}$ norms. Applying Theorem~3.1 of \citet{kuchibhotla2022moving} and integrating its tail bound gives, for fixed unit $u$ and $q\ge2$,
\[
 \left\|\frac1n\sum_i\left\{(u^\top V_i)^4
 -\mathbb E_{\mX}(u^\top V_i)^4\right\}\right\|_q
 \le C\left\{\sqrt{q/n}+q^2/n\right\}.
\]
Choose $q$ as a sufficiently large multiple of $\mathcal H_n(s_n^\star)$. Markov's inequality, the same nets and $\mathcal H_n(s_n^\star)^2=o(n)$ bound the empirical fourth moments uniformly with probability tending to one. Extending the seminorm $u\mapsto\{n^{-1}\sum_i|u^\top V_i|^4\}^{1/4}$ from the nets gives
\begin{equation}
 \sup_{\substack{|S|\le s_n^\star\\\|u\|_2=1}}
 \frac1n\sum_{i=1}^n(V_{i,S}^\top u)^4\le C.
 \label{supp-eq:supp-random-fourth-moment}
\end{equation}

Applying sub-Gaussian tails to the nets and all $n$ rows adds $\log n$ to the logarithmic union size. Thus, for $s\in\{s_0,s_n^\star\}$,
\begin{equation*}
 \max_{i\le n}\max_{|S|\le s}\|X_{i,S}\|_2
 \le C\{\mathcal H_n(s)\}^{1/2},
 \qquad
 \max_{i\le n}\max_{|S|\le s}\|V_{i,S}\|_2
 \le C\{\mathcal H_n(s)\}^{1/2}
\end{equation*}
with probability tending to one. Since $z_{i,S}=(\mF_S^0)^{-1/2}X_{i,S}$ and $\lambda_{\min}(\mF_S^0)\ge nc_F$, these bounds imply the stated rates for $q_\star(s)$ and $\ell_\star(s)$. For average spread, put $b=\sqrt n(\mF_S^0)^{-1/2}u$ for unit $u$. Then $\|b\|_2\le C$ and $z_{i,S}^\top u=n^{-1/2}X_{i,S}^\top b$. Cauchy--Schwarz and \eqref{supp-eq:supp-random-second-moment}--\eqref{supp-eq:supp-random-fourth-moment} give
\begin{align*}
 \sum_iw_i^0|z_{i,S}^\top u|^3
 &\le n^{-1/2}
 \left\{\frac1n\sum_i(V_{i,S}^\top b)^4\right\}^{1/2}
 \left\{\frac1n\sum_i(X_{i,S}^\top b)^2\right\}^{1/2}
 \lesssim n^{-1/2},\\
 \sum_i(w_i^0)^2(z_{i,S}^\top u)^4
 &=n^{-2}\sum_i(V_{i,S}^\top b)^4\lesssim n^{-1}.
\end{align*}
These are the uniform bounds in \eqref{eq:simple-spread-moments}. The H\"older argument in the proof of Proposition~\ref{prop:simple-fixed-design} gives $\mathfrak t_\star(s)+\mathfrak v_\star(s)\lesssim n^{-1/2}$ for $s\in\{s_0,s_n^\star\}$.
\end{proof}

\begin{proof}[Proof of Proposition~\ref{supp-prop:random-design-bridge}]
On the geometry event, use the deterministic bounds in \eqref{supp-eq:random-design-rates} to form $\eta_{\rm sel,n}$ and $\mathfrak k_n$. Put
\[
 A_n=\frac{s_n^\star+\log p}{\log p}
       \sqrt{\frac{H_\star s_n^\star\log p}{n}}.
\]
Condition \eqref{supp-eq:random-design-growth} is $A_n\to0$. Choose a common subpolynomial
$M_n=M_n^0\to\infty$ so slowly that $M_nA_n\to0$, and set
$r_n^0=M_n\sqrt{s_n^\star\log p}$ and $r_n=M_n\sqrt{s_0}$.
Then
\[
 q_\star(s_n^\star)r_n^0+
 \frac{\sqrt{s_n^\star}r_n^0}{\sqrt n}
 +q_\star(s_0)r_n\lesssim M_nA_n=o(1),
 \qquad r_n/r_n^0=O\{(\log p)^{-1/2}\}.
\]
The selection bound becomes
\begin{equation*}
 \eta_{\rm sel,n}\lesssim
 M_n\sqrt{\frac{s_n^\star\log p}{n}}
 +\frac{H_\star s_n^\star\log p}{n}
 +\sqrt{\frac{H_\star}{n}},
\end{equation*}
so $(s_0+\log p)\eta_{\rm sel,n}/\log p
=O(M_nA_n+A_n^2+A_n)=o(1)$.

For the oracle bound, $\mathcal H_n(s_0)\le H_\star$ and
$\log(en)\lesssim\log p$ give
\[
 s_0\mathfrak k_n\lesssim
 \frac{s_0}{\sqrt n}
 +\frac{s_0\sqrt{H_\star s_0\log(en)}}{n}
 +s_0^2\left(\frac{H_\star}{n}\right)^{3/2}\log(en)
 +\frac{\sqrt{H_\star}}{n}
 \lesssim A_n+A_n^2+A_n^3.
\]
Hence $s_0\mathfrak k_n=o(1)$, and
Proposition~\ref{supp-pr:truth-centred-verification} applies.
In Assumption~\ref{ass:score} take the deterministic envelope $d_n=C\sqrt{H_\star/n}$.
All bounds hold with fixed constants on the geometry event. Denote this event by $\mathcal X_n$.
The uniformity in that proposition therefore supplies a deterministic
$\delta_n\to0$ bounding the conditional response-failure probability for every
$\mX\in\mathcal X_n$. Since $\mathbb P_{\mX}(\mathcal X_n)\to1$, this proves the conditional statement and, by integration, its joint-law version.
\end{proof}

\begin{proof}[Proof of Proposition~\ref{supp-prop:poisson-subgaussian-bridge}]
Let $\mathcal U_s=\{u:\|u\|_2=1,\ \|u\|_0\le s\}$.
The first condition in \eqref{supp-eq:poisson-subgaussian-growth} implies
$H_\star^2=o(n)$ because $H_\star\lesssim s_n^\star\log p$.
The fourth-moment and row bounds in the proof of
Proposition~\ref{supp-prop:random-design-geometry}, applied to the unweighted rows, therefore give
\[
 \sup_{u\in\mathcal U_{s_n^\star}}\frac1n\sum_i(u^\top X_i)^4\le C,
 \qquad
 \max_{i\le n}\max_{|S|\le s}\|X_{i,S}\|_2
 \le C\sqrt{\mathcal H_n(s)},\quad s\in\{s_0,s_n^\star\},
\]
with probability tending to one. Moreover,
$\|\eta_i^0\|_{\psi_2}\le K_XB$ bounds all fixed exponential moments.
Independence and Chebyshev's inequality give
$n^{-1}\sum_i e^{r\eta_i^0}\le C$ for $r=2,4$ on the same type of event.
Cauchy--Schwarz and H\"older's inequality yield
\begin{equation}
 \sup_{u\in\mathcal U_{s_n^\star}}\frac1n\sum_i
 e^{\eta_i^0}(u^\top X_i)^2\le C,\qquad
 \sup_{u\in\mathcal U_{s_n^\star}}\frac1n\sum_i
 e^{\eta_i^0}|u^\top X_i|^3\le C.
 \label{supp-eq:supp-poisson-weighted-moments}
\end{equation}
For the cubic bound, use the fourth exponential moment and the empirical fourth moment of $u^\top X_i$.

The lower Fisher bound needs a separate argument. Choose a fixed $L_0$ so large that, uniformly over sparse unit $u$,
\[
 \frac1n\sum_i\mathbb E_{\mX}
 [(u^\top X_i)^2\mathbf1\{\eta_i^0<-L_0\}]\le c_X/2.
\]
This follows from Cauchy--Schwarz, bounded fourth moments and the sub-Gaussian tail of $\eta_i^0$.
The truncated rows $\widetilde X_i=X_i\mathbf1\{\eta_i^0\ge-L_0\}$ remain uniformly sub-Gaussian and have a sparse population second-moment lower bound $c_X/2$.
The sparse second-moment concentration argument in the preceding geometry proof applies to these independent rows, giving
\[
 \inf_{u\in\mathcal U_{s_n^\star}}\frac1n\sum_i(u^\top\widetilde X_i)^2
 \ge c_X/4
\]
with probability tending to one. Since $e^{\eta_i^0}\ge e^{-L_0}$ on the retained observations, this and \eqref{supp-eq:supp-poisson-weighted-moments} prove Assumption~\ref{ass:sparse-fisher} for
$\mF_S^0=\sum_i e^{\eta_i^0}X_{i,S}X_{i,S}^\top$.
The Fisher lower bound and the row bound give the influence estimate below.
To obtain average spread, apply the cubic bound in
\eqref{supp-eq:supp-poisson-weighted-moments} to the vector
$\sqrt n(\mF_S^0)^{-1/2}u$, whose norm is uniformly bounded, and then use H\"older's inequality for the mixed directions.
Thus there are design events $\mathcal X_n$, with
$\mathbb P_{\mX}(\mathcal X_n)\to1$, on which Assumption~\ref{ass:sparse-fisher} and
\begin{equation*}
 q_\star(s)\le C\sqrt{\mathcal H_n(s)/n},\qquad
 \mathfrak t_\star(s)\le Cn^{-1/2},\qquad s\in\{s_0,s_n^\star\},
\end{equation*}
hold with fixed constants.

For the conditional likelihood, the Poisson mgf calculation in
Subsection~\ref{supp-sec:six-model-cards} and \eqref{supp-eq:supp-score-domain} give
Assumption~\ref{ass:score} with the deterministic envelope $d_n=C\sqrt{H_\star/n}$.
Choose a common subpolynomial $M_n=M_n^0\to\infty$ so slowly that
\[
 M_n\left\{\sqrt{\frac{H_\star s_n^\star\log p}{n}}
 +\sqrt{\frac{s_n^\star(s_n^\star+\log p)^2}{n\log p}}\right\}\to0,
\]
and set $r_n^0=M_n\sqrt{s_n^\star\log p}$ and $r_n=M_n\sqrt{s_0}$.
The design bounds give \eqref{supp-eq:supp-model-radii}.
Since Poisson is canonical, \eqref{supp-eq:supp-canonical-rates} gives deterministic envelopes $\eta_{\rm sel,n}=CM_n\sqrt{s_n^\star\log p/n}$ and $\mathfrak k_n=Cn^{-1/2}$.
The choice of $M_n$ implies
\[
 \frac{(s_0+\log p)\eta_{\rm sel,n}}{\log p}\to0,\qquad
 s_0\mathfrak k_n\to0,\qquad
 \sqrt{s_0}\eta_{\rm sel,n}\to0.
\]
Proposition~\ref{supp-pr:truth-centred-verification} therefore gives Assumption~\ref{ass:truth-lan} and the likelihood part of Assumption~\ref{ass:oracle-laplace}, with deterministic $\bar\kappa_n^0=Cn^{-1/2}$.
The last limit also verifies the true-Fisher replacement condition.
The constants are uniform on $\mathcal X_n$, so the conditional response-failure probabilities tend to zero uniformly there, as required.
Neither $\ell_\star$ nor $\mathfrak v_\star$ enters this argument because the residual-information term vanishes.
\end{proof}

The assumptions of Proposition~\ref{supp-prop:poisson-subgaussian-bridge} do not require centred rows. Thus nonzero means and deterministic bounded coordinates are permitted under the stated sub-Gaussian and sparse second-moment bounds. In particular, the proposition applies to independent $N_p(0,\Sigma)$ rows when $\|\beta^0\|_2\le B$, $\|\Sigma\|_{\rm op}\le C_X$, the minimum eigenvalues of $\Sigma_{SS}$ are uniformly bounded away from zero over $|S|\le s_n^\star$, and \eqref{supp-eq:poisson-subgaussian-growth} holds.

\subsection{Prior and inferential consequences}
\label{supp-sec:model-verification-results}
\label{supp-sec:logistic-benchmark-comparison}

\begin{proof}[Proof of Corollary~\ref{cor:primitive-six-models}]
The fixed- or random-design criterion supplies Assumptions~\ref{ass:sparse-fisher}--\ref{ass:truth-lan} and the likelihood part of Assumption~\ref{ass:oracle-laplace}. Proposition~\ref{prop:slab-verification} supplies Assumption~\ref{ass:slab} at the same selection radius and the oracle slab condition in Assumption~\ref{ass:oracle-laplace}. With the imposed model-size prior and signal conditions, Theorems~\ref{thm:selection}--\ref{thm:oracle-bvm} and Corollaries~\ref{cor:posterior-contraction}--\ref{cor:selected-coverage} apply. For random designs, integrate the uniform conditional bounds over the good-design events. The frequentist-coverage conclusions retain the stated sampling-calibration conditions.
\end{proof}

\paragraph{Prior thresholds and the logistic benchmark}
For the fixed slabs, Proposition~\ref{prop:slab-verification} gives the polynomial truth-density lower bound, bounded selection-neighbourhood variation and oracle slab flatness. The radius choice \eqref{supp-eq:supp-fixed-design-radii} is sufficient for these conditions. The support-layer calculations in Lemma~\ref{supp-le:rec-support-control} therefore apply for $a_+>1$ and sufficiently large $K_0$, and the logarithmic deletion cost in \eqref{supp-eq:rec-deletion-cost} is $O(\log p)$ per missing coordinate. Thus $nb_0^2\ge C_\alpha\log p$ with sufficiently large $C_\alpha$ gives the required signal separation.
For the logistic benchmark, the fixed-scale Laplace slab and $\|\beta^0\|_\infty=O(\log p)$ meet these conditions with $a_-=a_+=a>1$. Proposition~\ref{supp-prop:primitive-bridge} verifies the likelihood assumptions, including the radius choices. Equation~\eqref{supp-eq:supp-fixed-calibration} supplies the true-Fisher replacement and growing-coordinate linearization condition. The moment bound below supplies Lyapunov's condition, completing the benchmark's inference claims.

Table~\ref{tab:closest-work} retains the remaining design and prior conditions of each comparator. In the notation of \citet{tang2024empirical}, bounded information factors permit $\tau=0$. Taking $\tau'=1/3$ and $s_n=\lfloor(n/\log p)^{1/3}\rfloor$, the rate $s_0^3\log p=o(n)$ gives $cs_0\le s_n$ eventually for every fixed $c>1$. This satisfies the support-size restrictions in their Condition~2 and Theorem~4, with their information assumptions imposed over that cutoff range. Their Theorem~5 gives the displayed Gaussian approximation under the same cubic rate. For \citet{lee2025advances}, retain Assumptions~(A1)--(A7) with $s_{\max}\asymp s_0$. In particular, their unweighted cubic-moment bound in (5.5) is imposed separately from our weighted average-spread conditions. Under the stated bounded factors, the second term in (5.6) gives $s_0^3\log p=o(n)$, the rate restriction in (A6) is weaker, and (5.20) gives the coordinatewise beta-min order. For our result, $s_0\le\log p$ makes $s_0^3=o(n\log p)$ follow from $s_0^2\log p=o(n)$. The table compares sufficient signal and sparsity conditions under these qualifications; each analysis retains its remaining assumptions.

\paragraph{Lyapunov verification}
Put $U_i=\tau_i^{-1}\xi'(\eta_i^0)(Y_i-\mu_i^0)$, so $\zeta_i=z_{i,S_0}U_i$.
For Gaussian, logistic and bounded-window probit regression,
$\mE_{\beta^0}|U_i|^3\le Cw_i^0$ follows from the Gaussian third moment or bounded responses and the stated predictor window. For upper-window Poisson regression,
$\mE_{\beta^0}|Y_i-\mu_i^0|^3\le\mu_i^0\sqrt{1+3\mu_i^0}\le Cw_i^0$.
Gamma scale invariance and bounded dispersion give the same normalized bound.
For negative-binomial regression, use
$U_i=\kappa(Y_i-\mu_i^0)/(\kappa+\mu_i^0)$ and
$\mE_{\beta^0}|Y_i-\mu_i^0|^3\le C\mu_i^0(1+\mu_i^0)^2$.
Consequently, Fisher normalization gives in all six cases
\[
 \sup_{\|v\|_2=1}\sum_i\mE_{\beta^0}|v^\top\zeta_i|^3
 \le Cq_\star(s_0)\sup_{\|v\|_2=1}
       \sum_iw_i^0(v^\top z_{i,S_0})^2
 =Cq_\star(s_0).
\]
Thus $q_\star(s_0)=o(1)$ verifies the Lyapunov conditions with $\delta=1$.
The fixed-design bound $q_\star(s_0)=O\{(s_0/n)^{1/2}\}$ makes this automatic for fixed $s_0$.
Growing-coordinate coverage additionally requires $\varepsilon_n\sqrt{s_0}=o(1)$.
The unbounded-weight Poisson route retains its separately imposed Lyapunov condition.

\end{document}